\documentclass[a4paper,twoside,11pt,reqno]{amsart}
\usepackage{fullpage}
\usepackage[T1]{fontenc}
\usepackage[utf8]{inputenc}
\usepackage{upgreek}
\usepackage{hyperref}
\usepackage{graphicx}
\usepackage{comment}
\usepackage[mathscr]{euscript}

\hypersetup{ocgcolorlinks=true,allcolors=testc}
\hypersetup{
     colorlinks   = true,
     citecolor    = black
}
\hypersetup{linkcolor=black}
\hypersetup{urlcolor=black}

\newcommand{\eqdef}{\stackrel{\mathrm{def}}{=}}

\newtheorem{theorem}{Theorem}
\newtheorem{lemma}[theorem]{Lemma}

\newtheorem{proposition}[theorem]{Proposition}

\newtheorem{conjecture}[theorem]{Conjecture}

\newtheorem*{definition*}{Definition}

\theoremstyle{remark}
\newtheorem{remark}[theorem]{Remark}

\newcommand{\R}{\mathbb{R}} 

\newcommand{\N}{\mathbb{N}}

\newcommand{\e}{\varepsilon}

\usepackage{dutchcal}

\renewcommand{\alpha}{\upalpha}
\renewcommand{\beta}{\upbeta}
\renewcommand{\gamma}{\upgamma}
\renewcommand{\delta}{\updelta}
\renewcommand{\zeta}{\upzeta}
\renewcommand{\eta}{\upeta}
\renewcommand{\theta}{\uptheta}
\renewcommand{\kappa}{\upkappa}
\renewcommand{\lambda}{\uplambda}
\renewcommand{\mu}{\upmu}
\renewcommand{\nu}{\upnu}
\renewcommand{\xi}{\upxi}
\renewcommand{\pi}{\uppi}
\renewcommand{\rho}{\uprho}
\renewcommand{\sigma}{\upsigma}
\renewcommand{\tau}{\uptau}
\renewcommand{\phi}{\upphi}
\renewcommand{\chi}{\upchi}
\renewcommand{\psi}{\uppsi}
\renewcommand{\omega}{\upomega}

\newcommand{\NN}{\mathsf{N}}

\newcommand{\BB}{\mathsf{B}}
\newcommand{\QQ}{\mathsf{Q}}

\newcommand{\E}{\mathbb{E}}

\newcommand{\mb}{\mathbb}

\newcommand{\mr}{\mathrm}

\newcommand{\vol}{\mathrm{vol}}

\newcommand{\dd}{\mathrm{d}}
\newcommand{\1}{\textbf{1}}

\newcommand{\p}[1]{\mathbb{P}\left( #1 \right)}
\newcommand{\scal}[2]{\!\left\langle #1, #2 \right\rangle\!}

\DeclareMathOperator{\sgn}{sgn}
\DeclareMathOperator{\proj}{Proj}

\title{Resilience of cube slicing in $\ell_p$}

\author{Alexandros Eskenazis}
\address{(A.~E.) CNRS, Institut de Math\'ematiques de Jussieu, Sorbonne Universit\'e, France and Trinity College, University of Cambridge, UK.}
\email{alexandros.eskenazis@imj-prg.fr, ae466@cam.ac.uk}

\author{Piotr Nayar}
\address{(P.~N.) University of Warsaw, 02-097 Warsaw, Poland.}
\email{nayar@mimuw@edu.pl} 

\author{Tomasz Tkocz}
\address{(T.~T.) Carnegie Mellon University, Pittsburgh, PA 15213, USA.}
\email{ttkocz@andrew.cmu.edu} 

\thanks{This material is based upon work supported by the NSF grant DMS-1929284 while A.~E.~was in residence at ICERM for the Harmonic Analysis and Convexity program. P.N.’s research was supported by the National Science Centre, Poland, grant 2018/31/D/ST1/0135. T.T.’s research was supported by the NSF grant DMS-2246484.}

\begin{document}

\maketitle
\vspace{-3mm}

\begin{abstract}
Ball's celebrated cube slicing theorem (1986) asserts that among hyperplane sections of the cube in $\R^n$, the central section orthogonal to $(1,1,0,\dots,0)$ has the greatest volume. We show that the same continues to hold for slicing $\ell_p$ balls when $p > 10^{15}$, as well as that the same hyperplane minimizes the volume of projections of  $\ell_q$ balls for $1 < q < 1 + 10^{-12}$. This extends Szarek's optimal Khinchin inequality (1976) which corresponds to $q=1$. These results thus address the resilience of the Ball--Szarek hyperplane in the ranges $2 < p < \infty$ and $1 < q < 2$, where analysis of the extremizers has been elusive since the works of Koldobsky (1998), Barthe--Naor (2002) and Oleszkiewicz (2003).

\end{abstract}

\bigskip

{\footnotesize

}


\section{Introduction}
Fix $p\in[1,\infty]$ and $n\in\N$. The present paper is devoted to the study of geometric parameters of the origin symmetric convex bodies
\begin{equation*}
\BB_p^n = \big\{x\in\R^n: \ \|x\|_p \leq 1\big\},
\end{equation*}
which are the closed unit balls of the normed spaces $\ell_p^n = (\R^n, \|\cdot\|_p)$, where for $p\in[1,\infty)$,
\begin{equation*}
\|x\|_p = \big( |x_1|^p+\ldots+|x_n|^p\big)^{1/p}
\end{equation*}
and $\|x\|_\infty = \max_{i=1,\ldots,n} |x_i|$, when $x=(x_1,\ldots,x_n)\in\R^n$. More specifically, we shall address the classical problem of identifying volume extremizing sections and projections of these bodies with respect to hyperplanes passing through the origin. This subject has attracted the interest of mathematicians for decades and a range of tools from probability and Fourier analysis have been employed in its study. We refer to the survey \cite{NT22} for a detailed account of classical results, recent advances and further references.


\subsection{Sections} Fix $p\in[1,\infty]$, $n\in\N$ and consider the following question for sections of $\BB_p^n$.
\smallskip
\begin{center}
\emph{Question 1. For which unit vectors $a$ in $\R^n$ is the volume of $\BB_p^n\cap a^\perp$ maximal or minimal?}
\end{center}
\smallskip
This problem and its variations has been intensively studied for five decades, since Hadwiger and Hensley showed in \cite{Had72,Hen79} that sections of the cube $\BB_\infty^n$ with coordinate hyperplanes $e_i^\perp$ have minimal volume. The reverse question of identifying the volume maximizing sections of the cube was answered the monumental work \cite{Bal86} of Ball, who proved that
\begin{equation} \label{eq:ball}
\mr{vol}\big( \BB_\infty^n \cap a^\perp\big) \leq \mr{vol}\big( \BB_\infty^n \cap \big(\tfrac{e_1+e_2}{\sqrt{2}}\big)^\perp\big).
\end{equation}

For $p<\infty$, the study of Question 1 was initiated by Meyer and Pajor. In \cite{MP88}, they extended the result of Hadwiger and Hensley by proving that sections of $\BB_p^n$ with coordinate hyperplanes $e_i^\perp$ have minimal volume for any $p\geq2$ and maximal volume when $p\in[1,2]$. In the reverse direction, they showed that when $p=1$, the section of the cross-polytope $\BB_1^n$ with the hyperplane orthogonal to $\tfrac{e_1+\cdots+e_n}{\sqrt{n}}$ has minimal volume, a result which was later extended to all values of $p\in[1,2]$ by Koldobsky \cite{Kol98} (see also \cite{ENT18} for a different probabilistic proof).

In view of the aforementioned results, the only missing case in the study of Question 1 is the identification of volume maximizing sections of $\BB_p^n$ when $p\in(2,\infty)$, a problem that has explicitly appeared in the literature multiple times \cite{Kol01, BN02, Ole03, KZ03, Kol05, MH11, Esk21, NT22}. In \cite{Ole03}, Oleszkiewicz made a crucial remark, showing that for $p\in(2,26)$ and $n$ large enough the section of $\BB_p^n$ with the hyperplane $\big( \tfrac{e_1+\cdots+e_n}{\sqrt{n}}\big)^\perp$ has in fact larger volume than the section with $\big(\tfrac{e_1+e_2}{\sqrt{2}}\big)^\perp$ and thus one cannot expect a Ball-type extremal for all $p>2$. In the same work, he speculated that Ball-type hyperplanes may maximize the volume of sections for \emph{sufficiently large} values of $p$. The first theorem of this work provides a positive answer to Oleszkiewicz's question.

\begin{theorem} \label{thm:sec}
There exists $26 < p_0 < 10^{15}$ such that for every $n\in\N$, $p\geq p_0$ and every unit vector $a$ in $\R^n$, we have
\begin{equation}
\vol\big(\BB_p^n\cap a^\perp\big) \leq \vol\big(\BB_p^n\cap\big(\tfrac{e_1+e_2}{\sqrt{2}}\big)^\perp\big).
\end{equation}
\end{theorem}

This is the first available result on maximal sections of $\BB_p^n$ for $p\in(2,\infty)$ and \emph{any} dimension $n\geq3$. A general conjecture for all choices of $p$ and $n$, predicting that the extremals undergo a phase transition, was proposed in \cite{Tko21} and \cite[Conjecture~2]{NT22}. Theorem \ref{thm:sec} partially confirms it. Let us formulate a more precise version of this conjecture.
\begin{conjecture}\label{conj:max-sec-Bpn}
For every $n \geq 3$, there is a unique $p_0(n)$ such that
\begin{equation}
\max_{a \in \mb{S}^{n-1}} \vol(\BB_p^n \cap a^\perp) = \begin{cases} \vol\big(\BB_p^n \cap \big(\frac{e_1+\ldots+e_n}{\sqrt{n}} \big)^\perp \big), & 2 < p \leq p_0(n), \\ \vol\big(\BB_p^n\cap\big(\tfrac{e_1+e_2}{\sqrt{2}}\big)^\perp\big), & p \geq p_0(n). \end{cases}
\end{equation}
Moreover,  $\lim_{n \to \infty} p_0(n) = 26.265...$ is the unique solution to the equation $2^{2/p} \Gamma (\frac{1}{p})^3 = \pi  p^2 \Gamma(\frac{3}{p})$ in the interval $(1,\infty)$.
\end{conjecture}
Let us remark that in the above conjecture the critical value $p=p_0(n)$ is given by the equation
\begin{equation}
	\vol\big(\BB_p^n \cap \big(\frac{e_1+\ldots+e_n}{\sqrt{n}} \big)^\perp \big) = \vol\big(\BB_p^n\cap\big(\tfrac{e_1+e_2}{\sqrt{2}}\big)^\perp\big).
\end{equation}
The limit of the ratio of these two volumes is equal to $\frac{2^{2/p} \Gamma (\frac{1}{p})^3}{\pi  p^2 \Gamma(\frac{3}{p})}$,  as was proved by Oleszkiewicz in \cite{Ole03} using Central Limit Theorem.


\subsection{Projections} Fix $q\in[1,\infty]$, $n\in\N$ and consider\mbox{ the dual question for projections of $\BB_q^n$.}
\smallskip
\begin{center}
\emph{Question 2. For which unit vectors $a$ in $\R^n$ is the volume of $\mr{Proj}_{a^\perp}\BB_q^n$ maximal or minimal?}
\end{center}
\smallskip
The current status of Question 2 is basically identical to that of Question 1. When $q=\infty$, Cauchy's projection formula shows that for every unit vector $a$, we have 
\begin{equation}
\mr{vol}\big(\mr{Proj}_{a^\perp} \BB_\infty^n\big) = \|a\|_1 \mr{vol}\big(\BB_\infty^{n-1}\big),
\end{equation}
which proves that the volume is minimized for $a=e_i$ and maximized for $a=\tfrac{e_1+\cdots+e_n}{\sqrt{n}}$. In the case of the cross-polytope $\BB_1^n$, similar reasoning based on Cauchy's formula (see \cite{Bal95}) shows that
\begin{equation} \label{eq:b1-formula}
\mr{vol}\big(\mr{Proj}_{a^\perp} \BB_1^n\big) = \frac{2^{n-1}}{(n-1)!} \mb{E}\Big|\sum_{j=1}^n a_j \e_j\Big|,
\end{equation}
where $\e_1,\e_2,\ldots$ is a sequence of independent symmetric $\pm 1$ random variables. Therefore, Jensen's inequality shows that $\mr{vol}\big(\mr{Proj}_{a^\perp} \BB_1^n\big)$ is maximal when $a=e_i$. In view of \eqref{eq:b1-formula}, identifying the volume minimizing projections of $\BB_1^n$ amounts to finding the sharp constant in the classical $L_1$-$L_2$ Khinchin inequality \cite{Khi22} which was famously discovered by Szarek. In geometric terms, the important result of \cite{Sza76} asserts that $\mr{vol}(\mr{Proj}_{a^\perp} \BB_1^n)$ is minimized for $a=\tfrac{e_1+ e_2}{\sqrt{2}}$.

The study of Question 2 for $1<q<\infty$ was initiated by Barthe and Naor in \cite{BN02}. In analogy to \cite{MP88}, they showed that projections of $\BB_q^n$ onto coordinate hyperplanes $e_i^\perp$ have minimal volume for $q\geq2$ and maximal volume for $q\in[1,2]$. Moreover, in the spirit of \cite{MP88, Kol98}, they proved that when $q\geq2$, the projections of $\BB_q^n$ onto the hyperplane orthogonal to $ \tfrac{e_1+\cdots+e_n}{\sqrt{n}}$ have maximal volume (see also \cite{KRZ04} for a different proof using the Fourier transform).

The volume minimizing hyperplane projections of $\BB_q^n$ remain unknown for $q\in(1,2)$. In analogy with Oleszkiewicz's observation \cite{Ole03} mentioned earlier, Barthe and Naor noticed that for $q\in\big(\tfrac{4}{3},2\big)$, the projection of $\BB_q^n$ onto the hyperplane $\big( \tfrac{e_1+\cdots+e_n}{\sqrt{n}}\big)^\perp$ has smaller volume than the projection onto $\big(\tfrac{e_1+e_2}{\sqrt{2}}\big)^\perp$ and thus one cannot expect a Szarek-type extremal for all $q\in[1,2)$. Our second theorem is the dual to Theorem \ref{thm:sec} and addresses Question 2 for $q$ near 1.

\begin{theorem} \label{thm:proj}
There exists $q_0\in\big(1+10^{-12},\tfrac{4}{3}\big)$ such that for every $n\in\N$, $q\in[1,q_0]$ and every unit vector $a$ in $\R^n$, we have
\begin{equation}
\vol\big(\mathrm{Proj}_{a^\perp}\BB_q^n\big) \geq \vol\big(\mathrm{Proj}_{\big(\tfrac{e_1+e_2}{\sqrt{2}}\big)^\perp} \BB_q^n\big).
\end{equation}
\end{theorem}

One can formulate a similar conjecture to the one for sections.
\begin{conjecture}\label{conj:min-proj-Bqn}
For every $n \geq 3$, there is a unique $q_0(n)$ such that
\begin{equation}
\min_{a \in \mb{S}^{n-1}} \vol(\mathrm{Proj}_{a^\perp}\BB_q^n) = \begin{cases} \vol(\mathrm{Proj}_{\big( \tfrac{e_1+\cdots+e_n}{\sqrt{n}}\big)^\perp} \BB_q^n \big), & q_0(n) < q \leq 2, \\ \vol\big(\mathrm{Proj}_{\big(\tfrac{e_1+e_2}{\sqrt{2}}\big)^\perp} \BB_q^n\big), & 1 \leq q \leq q_0(n). \end{cases}
\end{equation}
Moreover,  $\lim_{n \to \infty} q_0(n) = \frac43$.
\end{conjecture}


\subsection{Methods} \label{sec:methods}
The delicacy of, say, Theorem \ref{thm:sec} lies in the need to find a \emph{universal} $p_0$, independent of the unit vector $a$ and the dimension $n\in\N$, such that for every $p\geq p_0$,
\begin{equation} \label{eq:goal-sec-met}
\vol\big(\BB_p^n\cap a^\perp\big) \leq \vol\big(\BB_p^n\cap\big(\tfrac{e_1+e_2}{\sqrt{2}}\big)^\perp\big).
\end{equation}
On the other hand, finding such a $p_0(a)$ for a \emph{fixed} unit vector $a$ in $\R^n$ is an immediate consequence of the continuity of the section function $p\mapsto \vol(\BB_p^n\cap a^\perp)$, as the equality cases in Ball's inequality \eqref{eq:ball} are known to be only the vectors of the form $\tfrac{\pm e_i \pm e_j}{\sqrt{2}}$, where $i\neq j$. 

Let $a=(a_1,\ldots,a_n)$ be a unit vector and without loss of generality assume that its coordinates are positive and ordered, i.e.~$a_1\geq a_2\geq\ldots\geq a_n\geq0$. Choosing $p_0$ uniformly for \eqref{eq:goal-sec-met} to hold requires radically different arguments in the following ranges for $a$.

\medskip

\noindent {\it Case 1. The vector $a$ is far from the extremizer $\tfrac{e_1+e_2}{\sqrt{2}}$, say $\big|a-\tfrac{e_1+e_2}{\sqrt{2}} \big|\geq\delta_0$ for some $\delta_0>0$.}

\smallskip

Here the constant  $\delta_0$ depends on $p$ and $|\cdot|$ stands for the standard Euclidean norm.  The key ingredient in this range is the dimension-free stability of Ball's inequality \eqref{eq:ball} with respect to the unit vector $a$ which has been established in recent works \cite{CNT22,MR22} (see also Theorem \ref{thm:cnt} below for a statement with explicit constants). These works imply that, under the assumption of Case 1, there is a positive deficit in Ball's inequality. Building on the simple-minded argument based on continuity described above, one needs to reason that all functions of the form \mbox{$p\mapsto \vol(\BB_p^n\cap a^\perp)$} are equi-continuous at $p=\infty$ with a dimension-independent modulus. This strategy is implemented in Lemma \ref{lm:equicontinuity} and relies on a combination of Busemann's theorem \cite{Bus49} with a probabilistic formula expressing the volume of sections of $\BB_p^n$ as a negative moment of a sum of independent rotationally invariant random vectors in $\R^3$, following \cite{KK05,KK11,CKT21}.

\smallskip

\noindent {\it Case 2. The vector $a$ is near the extremizer $\tfrac{e_1+e_2}{\sqrt{2}}$, say $\big|a-\tfrac{e_1+e_2}{\sqrt{2}} \big|<\delta_0$.}

\smallskip

This range is evidently the more subtle one, as soft continuity-based arguments are deemed to fail near the equality case. In order to amend this, we introduce a novel inductive strategy. As our starting point, we express again the section function $\mr{vol}(\BB_p^n\cap a^\perp)$ as a negative moment of a sum of independent random variables. After a suitable application of Jensen's inequality, we use the inductive hypothesis according to which the desired inequality holds in dimension $n-2$ and this reduces the problem to an explicit two-dimensional estimate. Quite stunningly, the resulting estimate does not hold when the unit vector $a$ is far from the extremizer $\tfrac{e_1+e_2}{\sqrt{2}}$ and thus our inductive argument cannot circumvent the stability results which were crucially used in Case 1. Nevertheless, a delicate analysis allows us to deduce the technical estimate under the assumptions of Case 2 for $\delta_0$ small enough as a function of $p$ \mbox{and $p$ sufficiently large, thus proving Theorem \ref{thm:sec}.}

The proof of Ball's inequality \eqref{eq:ball} and its stability from \cite{CNT22} crucially use the Fourier transform representation for the volume of sections and properties of a certain special function. However, even in Ball's original proof \cite{Bal86}, the Fourier transform method is unable to analyze the case that the largest component $a_1$ of $a$ is greater than $\tfrac{1}{\sqrt{2}}$, which is instead handled by an elegant geometric argument. Unfortunately, a similar geometric argument applied to $\BB_p^n$ for $p<\infty$ does not yield the optimal bound \eqref{eq:goal-sec-met} for $a_1$ slightly larger than $\tfrac{1}{\sqrt{2}}$, which creates the need for a different method.  Surprisingly, our inductive approach outlined above does not use the Fourier transform directly, even though it uses Ball's inequality \eqref{eq:ball} and its stability as a black box. In a way, this method complements the Fourier analytic approach with a probabilistic component which permits an analysis near the extremizer.

The proof of Theorem \ref{thm:proj} relies on a very similar strategy apart from purely technical differences. In this case, the probabilistic representation for the volume of projections is due to \cite{BN02} and the stability of Szarek's inequality was obtained in \cite{DDS16}.


\section{Preliminaries}

In this section we present some probabilistic representations for the volume of sections and projections of $\BB_p^n$ (see also \cite{NT22} and the references therein) along with some crucial technical estimates which will be used in the proofs of Theorems \ref{thm:sec} and \ref{thm:proj}.


\subsection{Probabilistic representation of the volume of sections} In \cite{KK05}, Kalton and Koldobsky discovered an elegant probabilistic representation of the volume of sections of a convex set $K$ in $\R^n$ in terms of negative moments of a random vector $X$ uniformly distributed on $K$.
In the case of $K=\BB_p^n$, this representation takes the following explicit form (see \cite{CKT21} or \cite[Lemma~42]{NT22}).

\begin{lemma}
Fix $p\in[1,\infty)$, $n\in\N$ and let $Y_1,Y_2,\ldots$ be i.i.d.~random variables with density $e^{-\beta_p^p|x|^p}$, where $\beta_p = 2\Gamma\big(1+\tfrac{1}{p}\big)$. Then, for every unit vector $a$ in $\R^n$ we have
\begin{equation} \label{eq:CKT}
\frac{\mr{vol}(\BB_p^n\cap a^\perp)}{\mr{vol}(\BB_p^{n-1})} = \lim_{s\downarrow-1} \frac{1+s}{2} \mb{E} \Big| \sum_{j=1}^n a_j Y_j \Big|^{s}.
\end{equation}
When $p=\infty$, the same identity holds with $Y_1,Y_2,\ldots$ being i.i.d.~uniform on $[-1,1]$.
\end{lemma}

Using the representation \eqref{eq:CKT}, we derive the following crucial formula for our analysis.

\begin{proposition} \label{prop:sec-rep}
Fix $p\in[1,\infty)$ and $n\in\N$. Let $R_1,R_2,\ldots$ be i.i.d. positive random variables with density $\alpha_p^{-1} x^pe^{-x^p}{\bf 1}_{x>0}$, where $\alpha_p = \tfrac{1}{p}\Gamma\big(1+\tfrac{1}{p}\big)$ and $\xi_1,\xi_2,\ldots$ be i.i.d.~random vectors uniformly distributed on the unit sphere $\mb{S}^2$, independent of the random variables $R_i$.  Then, for every unit vector $a$ in $\R^n$ we have
\begin{equation} \label{eq:neg-mom}
\frac{\mr{vol}(\BB_p^n\cap a^\perp)}{\mr{vol}(\BB_p^{n-1})} = \Gamma\Big(1+\frac{1}{p}\Big)\ \mb{E} \Big| \sum_{j=1}^n a_j R_j\xi_j \Big|^{-1},
\end{equation}
where $|\cdot|$ denotes the Euclidean norm on the right-hand side. When $p=\infty$, the same identity holds with deterministic coefficients $R_1=\cdots=R_n=1$.
\end{proposition}

\begin{proof}
We shall assume that $p<\infty$ and the endpoint case follows (see also \cite{KK11}). Let $Y$ have density $e^{-\beta_p^p|x|^p}$, $R$ have density $\alpha_p^{-1}x^pe^{-x^p}{\bf 1}_{x>0}$ and $U$ be uniform on $[-1,1]$, independent of $R$. Then $Y$ has the same distribution as $\beta_p^{-1}RU$. More generally, if $V$ is a random variable with even density $g$ which is nonincreasing and of class $C^1$ on $(0,+\infty)$, then $V$ has the same distribution as $R_0U$, where $R_0$ has density $-2rg'(r)$ on $(0,\infty)$. Indeed, for $t>0$ we have
\begin{align*}
	\mb{P}\left\{R_0U>t\right\} & = \mb{P}\left\{U> \frac{t}{R_0}\right\} = \int_0^\infty \mb{P}\left\{U> \frac{t}{r}\right\} (-2r g'(r)) \dd r =  - \int_t^\infty \left(1-\frac{t}{r} \right) r g'(r) \dd r \\
	& = - \int_t^\infty \left(r-t \right) g'(r) \dd r =  \int_t^\infty g(r) \dd r = \mb{P}\left\{V>t\right\}.
\end{align*}
Therefore, \eqref{eq:CKT} can be rewritten as
\begin{equation} \label{eq:find-R}
\frac{\mr{vol}(\BB_p^n\cap a^\perp)}{\mr{vol}(\BB_p^{n-1})} = \lim_{s\downarrow-1} \frac{1+s}{2\beta_p^s} \mb{E} \Big| \sum_{j=1}^n a_j R_j U_j \Big|^{s}.
\end{equation}
By a result of K\"onig and Kwapie\'n \cite[Proposition~4]{KK01}, for every $x_1,\ldots,x_n\in\R$ and $s>-1$,
\begin{equation} \label{eq:kk01}
\mb{E}\Big| \sum_{j=1}^n x_j \xi_j\Big|^s = (1+s) \mb{E}\Big| \sum_{j=1}^nx_j U_j\Big|^s.
\end{equation}
Substituting \eqref{eq:kk01} in \eqref{eq:find-R} conditionally on $R_j$ and substituting the value of $\beta_p$ proves \eqref{eq:neg-mom}.
\end{proof}


\subsection{Probabilistic representation of the volume of projections}
The analogue of Proposition \ref{prop:sec-rep} for projections, expressing the normalized volume of projections of $\BB_q^n$ as an $L_1$-moment of a sum of independent random variables has been established in \cite[Proposition~2]{BN02}. 

\begin{proposition} [Barthe--Naor, \cite{BN02}] \label{prop:proj-rep}
Fix $q\in(1,\infty)$ and $n\in\N$. Let $X_1,X_2,\ldots$ be i.i.d.~random variables with density $\gamma_q^{-1} |x|^{\frac{2-q}{q-1}} e^{-|x|^{\frac{q}{q-1}}}$, where $\gamma_q =2(q-1)\Gamma\big(1+\tfrac{1}{q}\big)$. Then, for every unit vector $a$ in $\R^n$ we have
\begin{equation} \label{eq:bar-nao}
\frac{\mr{vol}(\mr{Proj}_{a^\perp}\BB_q^n)}{\mr{vol}(\BB_q^{n-1})} =\Gamma\Big(\frac{1}{q}\Big) \mb{E} \Big| \sum_{j=1}^n a_j X_j \Big|.
\end{equation}
When $q=1$, the identity reduces to the consequence \eqref{eq:b1-formula} of the Cauchy projection formula.
\end{proposition}


\subsection{Stability estimates}
As explained in the introduction, a crucial step in the proofs of Theorems \ref{thm:sec} and \ref{thm:proj} is a reduction to sections and projections with respect to hyperplanes near the extremizer $\big(\tfrac{e_1+e_2}{\sqrt{2}}\big)^\perp$. This will be a consequence of two recent works \cite{DDS16,CNT22} establishing the stability of the inequalities of Szarek \cite{Sza76} and Ball \cite{Bal86} with respect to the unit normal vector $a$. For the case of projections, we will use the following robust Szarek inequality proven in \cite{DDS16}.


\begin{theorem} [De--Diakonikolas--Servedio, \cite{DDS16}] \label{thm:dds}
There exists $\kappa_1>0$ such that for every $n\in\N$ and every unit vector $a$ in $\R^n$ with $a_1\geq\cdots\geq a_n\geq0$, we have
\begin{equation} \label{eq:dds}
\mb{E}\Big| \sum_{j=1}^n a_j \e_j\Big| \geq \frac{1}{\sqrt{2}} +\kappa_1 \Big| a- \frac{e_1+e_2}{\sqrt{2}}\Big|.
\end{equation}
We can take $\kappa_1=8\cdot 10^{-5}$ in this inequality.
\end{theorem} 

For the case of sections, we will use the following robust Ball inequality of \cite{CNT22}. We express it in the equivalent negative moment formulation which follows from Proposition \ref{prop:sec-rep}.

\begin{theorem} [Chasapis--Nayar--Tkocz, \cite{CNT22}] \label{thm:cnt}
There exists $\kappa_\infty>0$ such that for every $n\in\N$ and every unit vector $a$ in $\R^n$ with $a_1\geq\cdots\geq a_n\geq0$, we have
\begin{equation} \label{eq:dds}
\mb{E}\Big| \sum_{j=1}^n a_j \xi_j\Big|^{-1} \leq \sqrt{2} -\kappa_\infty \Big| a- \frac{e_1+e_2}{\sqrt{2}}\Big|.
\end{equation}
We can take $\kappa_\infty=6\cdot10^{-5}$ in this inequality.
\end{theorem}

Unfortunately, a direct implementation of the arguments of \cite{DDS16,CNT22} does not yield explicit values for the constants $\kappa_1$ and $\kappa_\infty$ which are needed for our estimation of $p_0$ and $q_0$ in Theorems \ref{thm:sec} and \ref{thm:proj}. In Section \ref{sec:stab}, we shall present a new short proof of Theorem \ref{thm:dds} which is in the spirit of \cite{CNT22} and gives the numerical constant $\kappa_1=8\cdot10^{-5}$. Moreover, we will explain how to quantify an existential argument used in \cite{CNT22} in order to prove Theorem \ref{thm:cnt} with $\kappa_\infty=6\cdot10^{-5}$.


\subsection{A technical lemma} In this section we present the following key lemma, which is crucial for the induction argument sketched in Section \ref{sec:methods} to work.

\begin{lemma}\label{lm:a1a2}
Let $c \geq 1$ and $p > 4\sqrt{2}c$. If $0<a_2\leq a_1$ satisfy  $\|(a_1,a_2)\|_p \leq 2^{\frac1p- \frac12}$ and $|a_i-\frac{1}{\sqrt{2}}|\leq \frac{c}{p}$ for $i = 1,2$, then we have
\begin{equation}
	|a_1-a_2| \leq  3.65  \sqrt{\frac{c}{p-2}} \sqrt{1-a_1^2- a_2^2}. 
\end{equation}
\end{lemma}

To prove it, we need an elementary inequality between $p$-means with a deficit.

\begin{lemma}\label{lm:p-means-deficit}
Let $\sigma > 0$, $r \geq \max\{\sigma,2\}$ and $b_1,b_2\in(0,1]$ with $1-  \frac{\sigma}{r} \leq \frac{b_2}{b_1} \leq 1$. Then, we have
\begin{equation} \label{eq:p-means-deficit}
	\left( \frac{b_1^r + b_2^r}{2} \right)^{\frac1r} \geq \frac{b_1+ b_2}{2} +(r-1) \frac{1-e^{-\frac{\sigma}{2}}}{4\sigma}|b_1-b_2|^2.
\end{equation}
\end{lemma}

\begin{proof}
Denote $c_r\eqdef(r-1) \frac{1-e^{-\frac{\sigma}{2}}}{4\sigma}$. Dividing both sides by $b_1$, introducing $\delta \eqdef 1- \frac{b_2}{b_1}$, raising the inequality to the power $r$ and using that $b_1\leq1$, we see that \eqref{eq:p-means-deficit} follows from
\[
	\frac{1+(1-\delta)^r}{2}  \geq  \left(1- \frac{\delta}{2} +c_r \delta^2 \right)^r, \qquad \delta \in \left[0,\frac{\sigma}{r}\right].
\]  
We have equality for $\delta=0$ and thus it is enough to show that on $[0,\frac{\sigma}{r}]$ the derivatives compare,
\[
	-\frac{r}{2}(1-\delta)^{r-1} \geq r\left(1- \frac{\delta}{2} +c_r \delta^2 \right)^{r-1}\left( -\frac12 +2c_r \delta \right).
\]
Multiplying both sides by $\frac2r$ and rearranging gives an equivalent form
\[
	1-4c_r \delta \geq \left( \frac{1-\delta}{1-\frac{\delta}{2} + c_r \delta^2} \right)^{r-1},
\]
since $1-\tfrac{\delta}{2}+c_r\delta^2>0$ on $\big[0,\tfrac{\sigma}{r}\big]$. To prove the last inequality, observe that
\[
\left( \frac{1-\delta}{1-\frac{\delta}{2} + c_r \delta^2} \right)^{r-1} \leq \left( \frac{1-\delta}{1-\frac{\delta}{2} } \right)^{r-1} \leq \left( 1-\frac{\delta}{2}  \right)^{r-1}.
\]
It is enough to check the inequality $\left( 1-\frac{\delta}{2}  \right)^{r-1} \leq 1-4c_r \delta$ only for $\delta \in \{0, \frac{\sigma}{r}\}$, since the left-hand side is convex in $\delta$. For $\delta = \frac{\sigma}{r}$ we have $( 1-\frac{\sigma}{2r})^{r-1} \leq e^{- \frac{\sigma}{2} \cdot \frac{r-1}{r}}$, so we would like to prove that 
\[
e^{- \frac{\sigma}{2} \cdot \frac{r-1}{r}} \leq 1- \frac{r-1}{r}(1-e^{-\frac{\sigma}{2}}).
\]
Since $u= \frac{r-1}{r} \in [0,1]$ we want to verify $e^{- \frac{\sigma}{2} u} \leq 1- u(1-e^{-\frac{\sigma}{2}})$, which follows by observing that the left-hand side is a convex function of $u$ and we have equality for $u \in \{0,1\}$. 
\end{proof}

\begin{proof}[Proof of Lemma \ref{lm:a1a2}]
Since $p>\sqrt{2}c$, we have
\[
	\frac{a_2}{a_1} \geq \frac{\frac{1}{\sqrt{2}}- \frac{c}{p}}{\frac{1}{\sqrt{2}}+ \frac{c}{p}} = \frac{1- \frac{\sqrt{2} c}{ p}}{1+ \frac{ \sqrt{2} c}{p}} \geq \left(1- \frac{ \sqrt{2} c}{ p} \right)^2 \geq 1 -  2\sqrt{2} \frac{c}{p},
\]
so $\frac{a_2^2}{a_1^2} \geq 1- 4\sqrt{2}\frac{c}{p} = 1- \frac{2\sqrt{2}c}{p/2}$.
We can apply Lemma \ref{lm:p-means-deficit} with $r=\tfrac{p}{2}$, $b_i=a_i^2$ and $\sigma=2\sqrt{2}c$ to get 
\[
	\frac12 \geq  \left( \frac{a_1^p + a_2^p}{2} \right)^{\frac2p} \geq \frac{a_1^2+ a_2^2}{2} + \left(\frac{p}{2}-1 \right) \frac{1-e^{-\sqrt{2}c}}{8\sqrt{2} c}|a_1^2-a_2^2|^2,
\] 
where the leftmost inequality is equivalent to $\|(a_1,a_2)\|_p \leq 2^{\frac1p- \frac12}$. By the assumptions, we also have $a_1+a_2 \geq \sqrt{2}-\frac{2c}{p} \geq \sqrt{2}-\frac{1}{2 \sqrt{2}}$ and $e^{-c \sqrt{2}}<e^{-\sqrt{2}}$. Therefore, rearranging gives
\[
	1-a_1^2-a_2^2 \geq \frac{c_0}{c} (p-2) |a_1-a_2|^2, \qquad c_0 = \frac{\left(\sqrt{2}- \frac{1}{2\sqrt{2}}\right)^2}{8 \sqrt{2}} (1- e^{-\sqrt{2}}).
\]
Thus, we conclude that
\[
	|a_1-a_2| \leq \frac{\sqrt{c}}{ \sqrt{c_0(p-2)}}  \sqrt{1-a_1^2- a_2^2}, \qquad \frac{1}{\sqrt{c_0}} < 3.65,
\]
which completes the proof.
\end{proof}


\section{Sections} \label{sec:proof-sec}


\subsection{Ancillary results}

We begin with a simple $L_2$-bound quantifying that the distribution of the random magnitudes $R_j$ from \eqref{eq:neg-mom} is close to the point mass at $1$ as $p$ gets large. Explicit computations using the density show that for every $s>-p-1$, the $s$-th moment of $R_1$ is
\begin{equation} \label{eq:R-mom}
\mb{E} R_1^s = \frac{\Gamma\big(1+\tfrac{s+1}{p}\big)}{\Gamma\big(1+\tfrac{1}{p}\big)}.
\end{equation} 

\begin{lemma}\label{lm:R-L2bound}
For $p > 5$, we have
\begin{equation} \label{eq:R-L2bound}
\E|R_1-1|^2 \leq \frac{2}{\Gamma(1+1/p)}p^{-2}.
\end{equation}
\end{lemma}
\begin{proof}
By \eqref{eq:R-mom},  we can write
\begin{equation*}
\E|R_1-1|^2 = \E R_1^2 -2\E R_1 + 1 = \frac{\Gamma(1+3/p)-2\Gamma(1+2/p) + \Gamma(1+1/p)}{\Gamma(1+1/p)}.
\end{equation*}
The function
\[
h(x) \eqdef \Gamma(1+3x)-2\Gamma(1+2x)+\Gamma(1+x)
\]
satisfies $h(0) = h'(0) = 0$, so for every $0<x < \frac15$, by Taylor's expansion with Lagrange's remainder, there exists $0 < \theta < x$ such that
\begin{equation} \label{eq:taylor}
h(x) = \frac{1}{2}x^2h''(\theta) = \frac{1}{2}x^2(9\Gamma''(1+3\theta) - 8\Gamma''(1+2\theta) + \Gamma''(1+\theta)).
\end{equation}
\begin{lemma}\label{lm:Gamma''}
The function $\Gamma''$ is decreasing on $(0, \frac85)$.
\end{lemma}
Taking this for granted, $\Gamma''(1+3\theta) < \Gamma''(1+2\theta)$ and $\Gamma''(s) < \Gamma''(1) =\gamma^2 + \frac{\pi^2}{6} < 2$ for $s\in(1,\frac85)$ (as usual $\gamma = 0.577..$ is the Euler-Mascheroni constant, and this calculation of $\Gamma''(1)$ can be done with the aid of the polygamma function, see for instance 6.4.2 in the standard reference \cite{AS64}). Equation \eqref{eq:taylor} thus gives $h(x) \leq 2x^2$. This applied to $x = \frac1p$ leads to \eqref{eq:R-L2bound}.
\end{proof}

\begin{proof}[Proof of Lemma \ref{lm:Gamma''}]
Note that clearly $\Gamma^{(4)} > 0$, so $\Gamma^{(3)}$ increases, so for $s \in (0, \frac{8}{5})$, we have $\Gamma^{(3)}(s) < \Gamma^{(3)}(\frac85) = -0.33..$ (to obtain such numerical values, we refer again to \cite{AS64}). Therefore, $\Gamma''$ is decreasing on $(0, \frac{8}{5})$.
\end{proof}

To deal with hyperplanes far from the extremizer, we will crucially rely on the equi-continuity of the section functions at $p=\infty$ which we will now verify. For $p\in[1,\infty]$ we introduce the normalized section function,
\begin{equation}\label{eq:A-def}
A_{n,p}(a) \eqdef \frac{\vol(\BB_p^n \cap a^\perp)}{\vol(\BB_p^{n-1})},
\end{equation}
where $a$ is a unit vector in $\R^n$. Additionally, observe that
\[
A_{n,\infty}(a) =  \frac{\vol(\BB_\infty^n \cap a^\perp)}{\vol(\BB_\infty^{n-1})} = \vol\big(\QQ_n \cap a^\perp\big),
\]
where $\QQ_n = [-\frac{1}{2},\frac12]^n$ is the unit-volume cube in $\R^n$. Recall that from Proposition \ref{prop:sec-rep},
\[
A_{n,p}(a) = \Gamma\Big(1+\frac{1}{p}\Big)\ \mb{E} \Big| \sum_{j=1}^n a_j R_j\xi_j \Big|^{-1}.
\]

\begin{lemma}\label{lm:equicontinuity}
Let $p > 5$. For every unit vector $a$ in $\R^n$, we have
\begin{equation}
\big|A_{n,p}(a) - A_{n,\infty}(a)\big| \leq \frac5p.
\end{equation}
\end{lemma}
\begin{proof}
First recall that for an arbitrary nonzero vector $x$ in $\R^n$,
\[
\NN(x) \eqdef \frac{|x|}{\vol(\QQ_n \cap x^\perp)} = \Big( \mb{E} \Big| \sum_{j=1}^n x_j\xi_j \Big|^{-1}\Big)^{-1}
\]
is a norm by Busemann's theorem \cite{Bus49}. In particular, using $1 \leq \vol(\QQ_n \cap x^\perp) \leq \sqrt{2}$, we get
\begin{align*}
\Big|\NN(y)^{-1} - \NN(x)^{-1}\Big| &= \frac{|\NN(x)-\NN(y)|}{\NN(x)\NN(y)} \leq \frac{\NN(x-y)}{\NN(x)\NN(y)}\\
&= \frac{|x-y|}{|x||y|}\frac{\vol(\QQ_n \cap x^\perp)\vol(\QQ_n \cap y^\perp)}{\vol(\QQ_n \cap (x-y)^\perp)} \leq 2 \frac{|x-y|}{|x||y|},
\end{align*}
where $x,y\in\R^n\setminus \{0\}$. Evoking \eqref{eq:neg-mom}, we can write
\[
\frac{A_{n,p}(a)}{\Gamma(1+1/p)} = \E_R\E_\xi\Big|\sum_{j=1}^n a_jR_j\xi_j\Big|^{-1} = \E_R\NN(aR)^{-1},
\]
where we use the ad hoc notation $aR$ for the vector $(a_1R_1, \dots, a_nR_n)$ in $\R^n$. From the previous bound on $1/\NN$, we thus obtain
\begin{align*}
\left|\frac{A_{n,p}(a)}{\Gamma(1+1/p)} - A_{n,\infty}(a)\right| &=  \left|\E \NN(aR)^{-1} - \NN(a)^{-1}\right|\leq 2\E\frac{|a-aR|}{|a|\cdot |aR|} = 2\E\frac{|a-aR|}{|aR|}.
\end{align*}
By the Cauchy-Schwarz inequality, 
\[
\E\frac{|a-aR|}{|aR|} \leq \sqrt{\E|a-aR|^2}\sqrt{\E|aR|^{-2}} = \sqrt{\E\sum_{j=1}^n a_j^2(R_j-1)^2}\sqrt{\E\Big(\sum_{j=1}^n a_j^2R_j^2\Big)^{-1}}.
\]
The first factor in the right-hand side is equal to $\|R_1-1\|_2$. By the convexity of the function $s\mapsto \tfrac{1}{s}$, 
$$\E\Big(\sum_{j=1}^n a_j^2R_j^2\Big)^{-1} \leq \sum_{j=1}^n a_j^2 \E R_j^{-2} \stackrel{\eqref{eq:R-mom}}{=} \frac{\Gamma\big(1-\frac1p\big)}{\Gamma\big(1+\frac1p\big)}.$$ 
Combining all the above, yields
\begin{align*}
\big|A_{n,p}(a) - \Gamma(1+1/p)A_{n,\infty}(a)\big| &\leq 2\|R_1-1\|_2\sqrt{\Gamma(1-1/p)\Gamma(1+1/p)}.
\end{align*}
Using Lemma \ref{lm:R-L2bound}, the right-hand side gets upper-bounded by
\[
2\sqrt{\frac{2}{\Gamma(1+1/p)}p^{-2}}\sqrt{\Gamma(1-1/p)\Gamma(1+1/p)} < \frac{2\sqrt{2\Gamma(1/2)}}{p} = \frac{2\sqrt{2} \sqrt[4]{\pi}}{p} 
\]
using $p>2$. Consequently,
\[
\big|A_{n,p}(a) - A_{n,\infty}(a)\big| \leq \frac{2\sqrt{2} \sqrt[4]{\pi}}{p} + \big(1-\Gamma(1+1/p)\big)A_{n,\infty}(a) \leq \frac{2\sqrt{2} \sqrt[4]{\pi}}{p} + \frac{\sqrt{2}\gamma}{p} < \frac5p,
\]
because $1-\Gamma(1+x) < -\Gamma'(1) x = \gamma x$ for $0 < x < 1$, by convexity of $\Gamma$ on $(0,\infty)$. Here, $\gamma = 0.577..$ is the Euler--Mascheroni constant.
\end{proof}


\subsection{Proof of Theorem \ref{thm:sec}}
Following notation \eqref{eq:A-def}, our goal is to prove that for every $p \geq p_0$ and every unit vector $a$ in $\R^n$, we have
\begin{equation}\label{eq:goalA}
A_{n,p}(a) \leq A_{n,p}\left(\frac{e_1+e_2}{\sqrt{2}}\right),
\end{equation}
where the right-hand side is explicitly given by
\[
A_{n,p}\left(\frac{e_1+e_2}{\sqrt{2}}\right) = \Gamma\Big(1+\frac{1}{p}\Big)\ \E\left|\frac{R_1\xi_1+R_2\xi_2}{\sqrt{2}}\right|^{-1} = A_{2,p}\left(\frac{e_1+e_2}{\sqrt{2}}\right)= \frac{1}{\|(\frac{1}{\sqrt{2}},\frac{1}{\sqrt{2}})\|_p} = 2^{\frac12-\frac1p}.
\]
Our proof will proceed by induction on $n$.
It is directly checked that the theorem holds when $n=2$, as $A_{2,p}(a) = \|a\|_p^{-1}$ for every unit vector $a$ in $\R^2$.
We therefore assume that $n \geq 3$ and $a_1 \geq \dots \geq a_n > 0$. Our analysis will differ \mbox{depending on the distance of $a$ to the extremizer. Let}
\begin{equation}
\delta(a) \eqdef \left|a - \frac{e_1+e_2}{\sqrt{2}}\right|^2 = 2-\sqrt{2}(a_1+a_2).
\end{equation}

\subsubsection{The vector $a$ is far from the extremizer}

Suppose that $\sqrt{\delta(a)} \geq \frac{c}{p}$ with $c=10^5$. Then, by the equi-continuity proven in Lemma \ref{lm:equicontinuity} and the stability of Ball's inequality from Theorem \ref{thm:cnt} with constant $\kappa_\infty=6\cdot10^{-5}$, we obtain
\[
A_{n,p}(a) \leq \frac{5}{p} + A_{n,\infty}(a) \leq \frac{5}{p} + \sqrt{2} - \kappa_\infty\sqrt{\delta(a)} \leq \sqrt{2} - \frac{\kappa_\infty c-5}{p}.
\]
Since 
\[
c \geq \frac{\sqrt{2}\log 2 + 5}{\kappa_\infty} = \frac{\sqrt{2}\log 2 + 5}{6}10^5 = 0.996..\cdot 10^5,
\]
we have
\[
\sqrt{2} - \frac{\kappa_\infty c-5}{p} \leq \sqrt{2}\left(1 - \frac{\log 2}{p}\right) \leq \sqrt{2}e^{-\frac{\log 2}{p}} = 2^{\frac12-\frac{1}{p}},
\]
which finishes the proof in this case (without using the inductive hypothesis) for $c=10^5$. \hfill$\Box$

\subsubsection{The vector $a$ is close to the extremizer}
Now, suppose that
\begin{equation*}\label{eq:da}
\sqrt{\delta(a)} < \frac{c}{p},
\end{equation*}
where $c=10^5$. This in particular implies that (as we already assume that $a_2 \leq a_1$),
\begin{equation*}
\frac{1}{\sqrt{2}} - \frac{c}{p} \leq a_2 \leq a_1 \leq \frac{1}{\sqrt{2}} + \frac{c}{p} .
\end{equation*}
Let us also notice for further use that $a_2\leq \frac{1}{\sqrt{2}}$ since $2a_2^2\leq a_1^2+a_2^2 \leq1$.  We shall consider $p > Lc+2$ for a large numerical constant $L$.  With hindsight, we put $L=7.9\cdot10^9$.  Observe that our goal \eqref{eq:goalA} is equivalent to the inequality
\begin{equation} \label{eq:goalA'}
\E\Big|\sum_{j=1}^n a_jR_j\xi_j\Big|^{-1} \leq C_p
\end{equation}
with
\begin{equation} \label{eq:formCp}
C_p = \E\left|\frac{R_1\xi_1+R_2\xi_2}{\sqrt{2}}\right|^{-1} = \frac{2^{\frac12-\frac1p}}{\Gamma(1+1/p)},
\end{equation} 
which we will now prove by induction on $n$.  We record for future estimates that when $p>Lc+2$, we have
\begin{equation} \label{eq:estimate-Cp}
1.41 < C_p  < 1.42,
\end{equation}
since $2^{10^{-6}} > 2^{1/p}\Gamma(1+1/p) > \Gamma(1+10^{-6})$.

Consider the random vectors $X = a_1R_1\xi_1 + a_2R_2\xi_2$ and $Y = \sum_{j>2} a_jR_j\xi_j$ in $\R^3$. Since $X$ and $Y$ are independent and rotationally invariant, the representation
\[
\E\Big|\sum_{j=1}^n a_jR_j\xi_j\Big|^{-1} = \E\min\big\{|X|^{-1},|Y|^{-1}\big\}
\]
holds (see, e.g., \cite[Lemma~6.6]{CNT22}). By the inductive hypothesis,
\[
\E|Y|^{-1} = \frac{1}{\sqrt{1-a_1^2-a_2^2}}\E\left|\frac{\sum_{j>2}a_jR_j\xi_j}{\sqrt{1-a_1^2-a_2^2}}\right|^{-1} \leq \frac{C_p}{\sqrt{1-a_1^2-a_2^2}},
\]
and hence (by the concavity and monotonicity of the function $t \mapsto \min\{|X|^{-1},t\}$), we get
\begin{equation} \label{eq:use-ind-and-jen}
\E\Big|\sum_{j=1}^n a_jR_j\xi_j\Big|^{-1} = \E\min\big\{|X|^{-1},|Y|^{-1}\big\}\leq \E\min\big\{|X|^{-1},\alpha^{-1}\big\},
\end{equation}
where we set
\begin{equation} \label{eq:forma}
\alpha \eqdef \frac{1}{C_p}\sqrt{1-a_1^2-a_2^2}.
\end{equation}
Observe that
\begin{equation} \label{eq:form-min}
\E\min\left\{|X|^{-1},\alpha^{-1}\right\} = \E|X|^{-1} - \E\big(|X|^{-1}-\alpha^{-1}\big)_+
\end{equation}
and
\begin{equation} \label{eq:mom-X-1}
\begin{split}
\E|X|^{-1} &= \frac{1}{\sqrt{a_1^2+a_2^2}}\E\left|\frac{a_1R_1\xi_1+a_2R_2\xi_2}{\sqrt{a_1^2+a_2^2}}\right|^{-1} \\
&\stackrel{\eqref{eq:neg-mom}}{=} \frac{1}{\sqrt{a_1^2+a_2^2}}\frac{1}{\left\|\frac{(a_1,a_2)}{\sqrt{a_1^2+a_2^2}}\right\|_p\Gamma(1+1/p)} = \frac{1}{\|(a_1,a_2)\|_p\Gamma(1+1/p)}.
\end{split}
\end{equation}
In view of the inductive step \eqref{eq:use-ind-and-jen} and the identities \eqref{eq:formCp}, \eqref{eq:form-min}, \eqref{eq:mom-X-1}, the desired inequality \eqref{eq:goalA'} is a consequence of the following proposition.

\begin{proposition} \label{prop:main}
Under the assumptions and notation above, for $p\geq 10^{15}$ we have
\begin{equation}\label{eq:goalA''}
\E\big(|X|^{-1}-\alpha^{-1}\big)_+ \geq C_p\left(\frac{2^{\frac1p-\frac12}}{\|(a_1,a_2)\|_p}-1\right).
\end{equation}
\end{proposition}

\begin{proof}
If the right-hand side is nonpositive, we are done. Otherwise,
\[
\|(a_1,a_2)\|_p < 2^{\frac1p-\frac12}.
\]
Since $\big|a_i-\tfrac{1}{\sqrt{2}}\big|<\tfrac{c}{p}$, Lemma \ref{lm:a1a2} gives
\begin{equation}\label{eq:a1-a2-alpha}
|a_1-a_2| \leq 3.65\sqrt{\frac{c}{p-2}}\sqrt{1-a_1^2-a_2^2} \stackrel{\eqref{eq:forma}}{=} 3.65\sqrt{\frac{c}{p-2}}C_p\alpha \stackrel{\eqref{eq:estimate-Cp}}{\leq} \frac{5.25\alpha}{\sqrt{L}}.
\end{equation}
To simplify, note that $\|(a_1,a_2)\|_p \geq 2^{1/p-1/2}\|(a_1,a_2)\|_2$, so
\begin{align*}
\frac{2^{\frac1p-\frac12}}{\|(a_1,a_2)\|_p}-1 \leq \frac{1}{\|(a_1,a_2)\|_2}-1=
\frac{1-(a_1^2+a_2^2)}{\sqrt{a_1^2+a_2^2}(1+\sqrt{a_1^2+a_2^2})} < \frac{C_p^2\alpha^2}{1.95},
\end{align*}
where we used that 
\begin{equation}\label{eq:a1a2squares}
a_1^2+a_2^2 \geq 2a_2^2 \geq 2\left(\frac{1}{\sqrt{2}}-\frac{c}{p}\right)^2 > 1-\frac{2\sqrt{2}c}{p} > 1 - \frac{2\sqrt{2}}{L} > 0.97
\end{equation}
and $\sqrt{u}(1+\sqrt{u}) > 1.95$ for $u > 0.97$. 
Since $\frac{C_p^3}{1.95} < \frac{1.42^3}{1.95} < \frac{3}{2}$, \eqref{eq:goalA''} will follow from
\begin{equation}\label{eq:goal}
\E\big(|X|^{-1}-\alpha^{-1}\big)_+ \geq \frac{3}{2}\alpha^2.
\end{equation}
Consider the event
\[
\mathcal{E} \eqdef \left\{R_1 \leq 1, \ |R_1-R_2| < \alpha, \ |a_1\xi_1+a_2\xi_2| < \frac14\alpha\right\}.
\]
On $\mathcal{E}$, we have
\begin{align*}
|X|  = |a_1R_1\xi_1 + a_2R_2\xi_2| &\leq |a_1R_1\xi_1 +a_2R_1\xi_2| + |a_2R_2\xi_2 - a_2R_1\xi_2|
\\ & = R_1|a_1\xi_1+a_2\xi_2| + a_2|R_2-R_1| < \frac{1}{4}\alpha + \frac{1}{\sqrt{2}}\alpha < \frac{24}{25}\alpha,
\end{align*}
so
\begin{equation} \label{eq:proba-esti}
 \E\big(|X|^{-1}-\alpha^{-1}\big)_+ \geq \frac{1}{24\alpha}\p{\mathcal{E}} = \frac{1}{24\alpha}\mb{P}\big\{R_1 \leq 1, \ |R_1-R_2| < \alpha\big\}\mb{P}\Big\{|a_1\xi_1+a_2\xi_2| < \frac14\alpha\Big\}.
\end{equation}
For the second probability in \eqref{eq:proba-esti}, observe that the random variable $|a_1\xi_1 + a_2\xi_2|^2$ has the same distribution as $a_1^2+a_2^2 + 2a_1a_2U$, with $U$ being uniform on $[-1,1]$. Therefore,
\[
\mb{P}\Big\{|a_1\xi_1+a_2\xi_2| < \frac14\alpha\Big\} = \mb{P}\Big\{U < \frac{\alpha^2/16-a_1^2-a_2^2}{2a_1a_2}\Big\}.
\]
Note that the condition
\begin{equation} \label{eq:condition}
-1 < \frac{\alpha^2/16-a_1^2-a_2^2}{2a_1a_2} < 1
\end{equation}
is equivalent to
\[
|a_1-a_2| < \frac{\alpha}{4} < a_1+a_2.
\]
The left inequality holds thanks to \eqref{eq:a1-a2-alpha}, provided that $L > (5.25\cdot 4)^2 = 441$, whereas the right inequality holds since $a_1+a_2 \geq 2a_2 > \sqrt{2} - \frac{2c}{p} > \sqrt{2} - \frac{2}{L} > 1.2$ which is greater than $\frac{\alpha}{4}$ since
\begin{equation}\label{eq:alpha-upbd}
\alpha \leq \frac{1}{C_p}\sqrt{1-2a_2^2} \stackrel{\eqref{eq:estimate-Cp}}{<} \frac{1}{1.41}\sqrt{1-2\left(\frac{1}{\sqrt{2}}-\frac{c}{p}\right)^2} < \frac{1}{1.41}\sqrt{2\sqrt{2}\frac{c}{p}} < \frac{1.2}{\sqrt{L}}.
\end{equation}
As \eqref{eq:condition} holds, we have
\[
\mb{P}\Big\{|a_1\xi_1+a_2\xi_2| < \frac14\alpha\Big\} = \frac12\left(\frac{\alpha^2/16-a_1^2-a_2^2}{2a_1a_2} + 1\right) = \frac{\alpha^2/16 - (a_1-a_2)^2}{4a_1a_2}.
\]
Using \eqref{eq:a1-a2-alpha} and the estimate $4a_1a_2 \leq 2(a_1^2+a_2^2) < 2$, we get
\begin{equation} \label{eq:proba1}
\mb{P}\Big\{|a_1\xi_1+a_2\xi_2| < \frac14\alpha\Big\} > \frac{1-441L^{-1}}{32}\alpha^2.
\end{equation}
For the other probability in \eqref{eq:proba-esti}, it is convenient to place a uniform function of constant mass under the density of $R_1$, which is doable due to the following technical lemma.
\phantom\qedhere
\end{proof}

\begin{lemma}\label{lm:unif-under}
Fix $p\in(1,\infty)$ and let $g_p:\R_+\to\R_+$ be the density of $R_1$. Then, we have
\begin{equation}
\forall \ x>0, \qquad {g}_p(x) \geq \frac{p}{4}\1_{[1-\frac{1}{2p},1]}(x).
\end{equation}
\end{lemma}
\begin{proof}
Recall from Proposition \ref{prop:sec-rep} that $g_p(x) = p\Gamma(1+1/p)^{-1} x^p e^{-x^p}$ for $x>0$. Since $g_p$ is log-concave, it suffices to check the inequality at the endpoints $x = 1-\frac{1}{2p}$ and $x = 1$. For the first endpoint, we have
\begin{align*}
{g}_p\big(1-\tfrac{1}{2p}\big) = \frac{p}{\Gamma(1+1/p)}\big(1-\tfrac{1}{2p}\big)^pe^{-(1-\frac{1}{2p})^p} > \frac{p}{2}e^{-e^{-1/2}} > \frac{p}{4},
\end{align*}
using that $(1-\frac{1}{2p})^p < e^{-1/2}$ and $(1-\frac{1}{2p})^p > \frac{1}{2}$. Moreover, for the second endpoint, 
\[
{g}_p(1) = \frac{p}{e\Gamma(1+1/p)} > \frac{p}{4}.\qedhere
\]
\end{proof}

\begin{proof} [Finishing the proof of Proposition \ref{prop:main}]
We estimate the first probability in \eqref{eq:proba-esti} using Lemma \ref{lm:unif-under},
\begin{equation} \label{eq:proba2}
\begin{split}
\mb{P}\big\{R_1 \leq 1, \ |R_1-R_2| <\alpha\big\} &\geq \int\!\!\!\!\int_{\{x \leq 1, \ |x-y| < \alpha\}}\left(\frac{p}{4}\right)^2\1_{[1-\frac{1}{2p},1]}(x)\1_{[1-\frac{1}{2p},1]}(y)\, \dd x\, \dd y \\
&= \begin{cases} \frac{1}{64}, & \mbox{if } \alpha > \frac{1}{2p} \\ \frac{p^2\alpha}{16}\left(\frac{1}{p}-\alpha\right), & \mbox{if } \alpha \leq \frac{1}{2p} \end{cases},
\end{split}
\end{equation}
where the equality is an elementary computation. In the case $\alpha\leq\tfrac{1}{2p}$, we further have $\frac{1}{p}-\alpha \geq \frac{1}{2p}$, so the probability is further bounded from below by $\frac{p\alpha}{32}$, which we will use.

\noindent $\bullet$ If $\alpha > \frac{1}{2p}$, inequalities \eqref{eq:proba-esti}, \eqref{eq:proba1} and \eqref{eq:proba2} yield the lower bound
\[
 \E(|X|^{-1}-\alpha^{-1})_+ \geq \frac{1}{24\alpha}\cdot\frac{1}{64}\cdot\frac{1-441L^{-1}}{32}\alpha^2 = \left(\frac{1-441L^{-1}}{2^{14}\cdot 3}\cdot\frac{1}{\alpha}\right)\alpha^2 \stackrel{\eqref{eq:alpha-upbd}}{>} \left(\frac{1-441L^{-1}}{2^{14}\cdot 3\cdot 1.2}\sqrt{L}\right)\alpha^2.
\]
Since $L=7.9\cdot10^9$, this gives the desired bound \eqref{eq:goal} by $\frac{3}{2}\alpha^2$. 

\noindent $\bullet$ If $\alpha \leq \frac{1}{2p}$, inequalities \eqref{eq:proba-esti}, \eqref{eq:proba1} and \eqref{eq:proba2} yield the lower bound
\[
 \E(|X|^{-1}-\alpha^{-1})_+ \geq \frac{1}{24\alpha}\cdot \frac{p\alpha}{32}\cdot\frac{1-441L^{-1}}{32}\alpha^2 = \frac{p(1-441L^{-1})}{2^{13}\cdot 3}\alpha^2 > \frac{(L-441)c}{2^{13}\cdot 3}\alpha^2.
\]
This is at least $\frac{3}{2}\alpha^2$ for the  chosen $L$, which completes the proof of \eqref{eq:goal} for $p\geq p_0$, where
\[
p_0 = Lc + 2 < 8\cdot 10^9\cdot 10^5 < 10^{15}.\qedhere
\]
\end{proof}


\section{Projections}\label{sec:proof-proj}

The proof here parallels the one from Section \ref{sec:proof-sec}. For the readers' convenience, we include all the details (which are in fact easier in several places).

\subsection{Ancillary results}

We start by quantifying how close the distribution of the $X_j$ from \eqref{eq:bar-nao} is to that of a Rademacher variable (in the Wasserstein-2 distance). Explicit computations using the density show that for every $s>-\tfrac{1}{q-1}$, the $s$-th moment of $|X_1|$ is
\begin{equation} \label{eq:X-mome}
\mb{E}|X_1|^s = \frac{\Gamma\big(1+\frac{(s-1)(q-1)}{q}\big)}{\Gamma\big(\frac{1}{q}\big)}.
\end{equation}

\begin{lemma}\label{lm:coupling}
For $1 < q < 2$, we have
\begin{equation}
\E|X_1-\sgn(X_1)|^2 \leq 9\Big(1-\frac1q\Big)^2.
\end{equation}
\end{lemma}
\begin{proof}
Observe that
\[
\E|X_1-\sgn(X_1)|^2 = \E X_1^2 - 2\E|X_1| + 1 \stackrel{\eqref{eq:X-mome}}{=} \frac{\Gamma(2-1/q) - 2 + \Gamma(1/q)}{\Gamma(1/q)}.
\]
Since $\Gamma$ is decreasing on $(0,1)$, $\Gamma(1/q) \geq \Gamma(1) = 1$. Using Taylor's expansion with Lagrange's remainder, for every $0 < x  < 1$ there exists $0<\theta<x$ such that
\[
h(x) \eqdef \Gamma(1-x) + \Gamma(1+x) -2 = \frac12x^2\big(\Gamma''(1-\theta)+\Gamma''(1+\theta)\big).
\]
Thus for $0 < x < 1/2$, we have
\[
h(x) \leq \frac12x^2\big(\|\Gamma''\|_{L_\infty(\frac12,1)} + \|\Gamma''\|_{L_\infty(1,\frac32)}\big) = \frac{1}{2}x^2\big(\Gamma''(1/2) + \Gamma''(1)\big) < 9x^2
\]
since $\Gamma''$ decreases on $(\frac12, \frac32)$, by Lemma \ref{lm:Gamma''}. Applying this to $x=1-\tfrac{1}{q}$, we indeed obtain
\[
\E|X_1-\sgn(X_1)|^2 \leq h\Big(1-\frac1q\Big) \leq 9\Big(1-\frac1q\Big)^2. \qedhere
\]
\end{proof}

From this estimate, we can easily deduce the equi-continuity of the normalized projection functions, which we state directly in probabilistic terms in view of Proposition \ref{prop:proj-rep}.

\begin{lemma}\label{lm:equicontinuity-proj}
Let $1 < q  < 2$, $X_1, X_2, \dots$ be i.i.d.~random variables from \eqref{eq:bar-nao} and $\e_1, \e_2, \dots$ be i.i.d.~Rademacher random variables. For every unit vector $a$ in $\R^n$, we have
\begin{equation}
\Bigg|\E\Big|\sum_{j=1}^n a_jX_j\Big| - \E\Big|\sum_{j=1}^n a_j\e_j\Big|\Bigg| \leq 3\Big(1-\frac1q\Big).
\end{equation}
\end{lemma}
\begin{proof}
Since $\e_j$ has the same distribution as $\sgn(X_j)$, we have
\begin{equation*}
\begin{split}
\Bigg|\E\Big|  \sum_{j=1}^n a_j & X_j\Big| - \E\Big|\sum_{j=1}^n a_j\e_j\Big|\Bigg| = \Bigg|\E\Big|\sum_{j=1}^n a_jX_j\Big| - \E\Big|\sum_{j=1}^n a_j\sgn(X_j)\Big|\Bigg| \\
&\leq \E\Big|\sum_{j=1}^n a_j\big(X_j - \sgn(X_j)\big)\Big| \leq \sqrt{\E\Big|\sum_{j=1}^n a_j\big(X_j - \sgn(X_j)\big)\Big|^2} = \sqrt{\E|X_1-\sgn(X_1)|^2}
\end{split}
\end{equation*}
and Lemma \ref{lm:coupling} finishes the proof.
\end{proof}


\subsection{Proof of Theorem \ref{thm:proj}}

By virtue of \eqref{eq:bar-nao}, our goal is to show that for every $1 < q < q_0$ and every unit vector $a$ in $\R^n$, we have
\begin{equation}\label{eq:proj-goal}
\E\Big|\sum_{j=1}^n a_jX_j \Big| \geq \mb{E} \Big| \frac{X_1+X_2}{\sqrt{2}}\Big|\eqdef c_q.
\end{equation}
For later use, we note that thanks to \eqref{eq:bar-nao}, for every vector $a$ in $\R^2$,
\begin{equation}\label{eq:proj-n2}
\begin{split}
\E|a_1X_1+a_2X_2| = |a|\frac{\vol(\proj_{a^\perp}\BB_q^2)}{2\Gamma(1/q)} &= \frac{|a|}{\Gamma(1/q)}\sup_{x \in \partial \BB_q^2}\left|\scal{x}{\frac{1}{|a|}(-a_2,a_1)}\right| = \frac{\|a\|_{\frac{q}{q-1}}}{\Gamma(1/q)}.
\end{split}
\end{equation}
In particular, we have
\begin{equation} \label{eq:form-cq}
c_q =  \mb{E} \Big| \frac{X_1+X_2}{\sqrt{2}}\Big| = \frac{2^{\frac12-\frac1q}}{\Gamma(1/q)}.
\end{equation}
In view of the above explicit expression, inequality \eqref{eq:proj-goal} clearly holds for $n=2$. We therefore assume that $n \geq 3$, $a_1 \geq \dots \geq a_n > 0$ and proceed by induction on $n$. Recall the definition of the deficit parameter used earlier,
\[
\delta(a) = \left|a - \frac{e_1+e_2}{\sqrt{2}}\right|^2 = 2-\sqrt{2}(a_1+a_2).
\]

\subsubsection{The vector $a$ is far from the extremizer}

Here we consider the case \mbox{$\sqrt{\delta(a)} \geq \big(1-\tfrac1q\big)c$} for a numerical constant $c > 0$.   With hindsight, we set
$$c=\frac{5-\sqrt{2}}{8}\cdot10^5.$$
Using the equi-continuity from Lemma \ref{lm:equicontinuity-proj} and the robust version of Szarek's inequality from Theorem \ref{thm:dds}, we obtain
\[
\E\Big|\sum_{j=1}^n a_jX_j\Big| \geq \E\Big|\sum_{j=1}^n a_j\e_j\Big| - 3\Big(1-\frac1q\Big) \geq \frac{1}{\sqrt{2}} + \kappa_1\sqrt{\delta(a)} - 3\Big(1-\frac1q\Big) \geq \frac{1}{\sqrt2} + (\kappa_1c - 3)\Big(1-\frac1q\Big).
\]
Note that by convexity, $2^x < 1 + 2(\sqrt{2}-1)x$ for $0 < x  < \frac12$, which with $x =  1-\frac1q$ gives
\[
c_q = \frac{2^{\frac12-\frac1q}}{\Gamma(1/q)} \leq 2^{\frac12-\frac1q} = \frac{1}{\sqrt{2}}2^{1-\frac1q} < \frac{1}{\sqrt{2}} + (2-\sqrt{2})\Big(1-\frac1q\Big).
\]
Since
\[
c = \frac{5-\sqrt{2}}{\kappa_1} = \frac{5-\sqrt{2}}{8}\cdot 10^5,
\]
we get the desired bound \eqref{eq:proj-goal} (nota bene, without the inductive argument). \hfill$\Box$

\subsubsection{The vector $a$ is close to the extremizer}
It is left to consider the case when
\begin{equation*}\label{eq:da-proj}
\sqrt{\delta(a)} < \Big(1-\frac1q\Big)c,
\end{equation*}
where $c= \frac{5-\sqrt{2}}{8} \cdot 10^5$, as in the previous case. In particular, we also have that
\begin{equation}\label{eq:a-proj}
\frac{1}{\sqrt{2}} - c\Big(1-\frac1q\Big) \leq a_2 \leq a_1 \leq \frac{1}{\sqrt{2}} + c\Big(1-\frac1q\Big).
\end{equation}
Letting $p=\tfrac{q}{q-1}$, we shall assume that $p$ is large relative to $c$, say $p>Lc+2$ with a positive numerical constant $L$, with hindsight set to be
$$L=8 \ 294\ 400. $$ 
In particular, when $\frac1p = 1 - \frac1q < 10^{-5}$,
\begin{equation} \label{eq:estim-cq}
0.7 < c_q< 0.71,
\end{equation}
since $0.7 < \frac{2^{-1/2}}{\Gamma(1-10^{-5})} < \frac{2^{1/2-1/q}}{\Gamma(1/q)} < 2^{-1/2+10^{-5}} < 0.71$.

To run an inductive argument in order to prove \eqref{eq:proj-goal}, we consider the random variables \mbox{$X = a_1X_1 + a_2X_2$} and $Y = \sum_{j>2} a_jX_j$. By the independence and symmetry of $X$ and $Y$,
\[
\E\Big|\sum_{j=1}^n a_jX_j\Big| = \E|X+Y| = \E\max\{|X|,|Y|\}.
\]
Using the inductive hypothesis,
\begin{equation*}
\E|Y| = \sqrt{1-a_1^2-a_2^2}\E\left|\frac{\sum_{j>2}a_jX_j}{\sqrt{1-a_1^2-a_2^2}}\right| \geq c_q\sqrt{1-a_1^2-a_2^2},
\end{equation*}
hence (by the convexity of the nondecreasing function $t \mapsto \max\{|X|,t\}$), we get
\begin{equation} \label{eq:ind-proj}
\E\Big|\sum_{j=1}^n a_jX_j\Big| = \E\max\left\{|X|,|Y|\right\} \geq \E\max\left\{|X|,\alpha\right\},
\end{equation}
where we set
\begin{equation} \label{eq:a-proje}
\alpha \eqdef c_q\sqrt{1-a_1^2-a_2^2}.
\end{equation}
Observe that
\begin{equation} \label{eq:form-max}
\E\max\left\{|X|,\alpha\right\} = \E|X| + \E(\alpha-|X|)_+
\end{equation}
and, by \eqref{eq:proj-n2},
\begin{align} \label{eq:form-|X|}
\E|X| = \frac{\|(a_1,a_2)\|_{\frac{q}{q-1}}}{\Gamma(1/q)} \stackrel{\eqref{eq:form-cq}}{=} c_q2^{\frac1q-\frac12}\|(a_1,a_2)\|_{\frac{q}{q-1}}.
\end{align}
In view of the inductive step \eqref{eq:ind-proj} and the identities \eqref{eq:form-max} and \eqref{eq:form-|X|}, the desired inequality \eqref{eq:proj-goal} is a consequence of the following proposition.

\begin{proposition} \label{prop:main-proj}
Under the assumptions and notation above, for $1\leq q \leq1+10^{-12}$ we have
\begin{equation}\label{eq:aX-proj}
\E\big(\alpha-|X|\big)_+ \geq \Big(1-2^{\frac1q-\frac12}\|(a_1,a_2)\|_{\frac{q}{q-1}}\Big)c_q.
\end{equation}
\end{proposition}

\begin{proof}
If the right-hand side is nonpositive, we are done. Otherwise,
\[
\|(a_1,a_2)\|_{\frac{q}{q-1}} < 2^{\frac12-\frac1q}.
\]
Letting $p = \frac{q}{q-1}$ and recalling \eqref{eq:a-proj}, we see that we can apply Lemma \ref{lm:a1a2} to conclude that
\begin{equation}\label{eq:a1-a2-alpha-proj}
|a_1-a_2| \leq 3.65\sqrt{\frac{c}{p-2}}\sqrt{1-a_1^2-a_2^2} \stackrel{\eqref{eq:a-proje}}{=} \frac{3.65}{c_q}\sqrt{\frac{c}{p-2}}\alpha \stackrel{\eqref{eq:estim-cq}}{<} \frac{5.25\alpha}{\sqrt{L}}.
\end{equation}
To simplify the right-hand side of \eqref{eq:aX-proj}, we write
\begin{align*}
c_q\Big(1-2^{\frac1q-\frac12}\|(a_1,a_2)\|_{\frac{q}{q-1}}\Big) &\leq c_q\big(1 - \|(a_1,a_2)\|_2\big) = c_q\frac{1-(a_1^2+a_2^2)}{1+\sqrt{a_1^2+a_2^2}} \\
&\stackrel{\eqref{eq:a-proje}}{=}\frac{\alpha^2}{c_q(1+\sqrt{a_1^2+a_2^2})} \stackrel{\eqref{eq:estim-cq}}{<} \frac{\alpha^2}{0.7(1+\sqrt{0.97})} < \frac{3}{4}\alpha^2,
\end{align*}
as we have $a_1^2+a_2^2 > 0.97$, see \eqref{eq:a1a2squares}. Therefore, it suffices to show that
\begin{equation}\label{eq:ultgoal-proj}
\E(\alpha-|X|)_+ \geq \frac{3}{4}\alpha^2.
\end{equation}

Since the $X_j$ are symmetric random variables,  each $X_j$ has the same distribution as $\e_j|X_j|$, for independent random signs $\e_j$,  also independent of all the other random variables.  We consider the event
\[
\mathcal{E} \eqdef \left\{|X_1| \leq 1, \ \big||X_1|-|X_2|\big| < \alpha, \ |a_1\e_1+a_2\e_2| < \frac14\alpha\right\},
\]
on which we have
\[
\big|a_1\e_1|X_1| + a_2\e_2|X_2|\big| \leq |X_1||a_1\e_1+a_2\e_2| + a_2\big||X_2|-|X_1|\big| < \frac{1}{4}\alpha + \frac{1}{\sqrt{2}}\alpha < \frac{24}{25}\alpha
\]
and thus, since $X$ has the same distribution as $a_1\e_1|X_1|+a_2\e_2|X_2|$,  we obtain the lower bound
\begin{equation} \label{eq:proba-projec}
\E(\alpha-|X|)_+ \geq \frac{\alpha}{25}\p{\mathcal{E}} = \frac{\alpha}{25}\mb{P}\big\{|X_1| \leq 1, \ \big||X_1|-|X_2|\big| < \alpha\big\}\mb{P}\Big\{|a_1\e_1+a_2\e_2| < \frac14\alpha\Big\}.
\end{equation}
The second probability in \eqref{eq:proba-projec} is clearly at least $\frac12$ provided that
\[
|a_1-a_2| < \frac{\alpha}{4}.
\]
For this to hold, it suffices that $L > (5.25\cdot 4)^2 = 441$, by virtue of \eqref{eq:a1-a2-alpha-proj}.
For the first probability, analogously to Lemma \ref{lm:unif-under}, we will place a constant function under the density $f_q$ of $|X_1|$.
\phantom\qedhere
\end{proof}

\begin{lemma}\label{lm:unif-under-proj}
Fix $q\in(1,\frac32)$ and let $f_q:\R_+\to\R_+$ be the density of $|X_1|$. Then, we have
\begin{equation}
f_q(x) \geq \frac{1}{4(q-1)}\1_{[1-\frac{q-1}{2},1]}(x), \qquad x > 0.
\end{equation}
\end{lemma}
\begin{proof}
Recall from Proposition \ref{prop:proj-rep},
\[
f_q(x) = \frac{1}{(q-1)\Gamma(1+\frac1q)}x^{\frac{2-q}{q-1}}e^{-x^{\frac{q}{q-1}}}, \qquad x > 0.
\]
The proof is almost identical to that of Lemma \ref{lm:unif-under}. It suffices to check that the inequality holds for $x =1-\frac{q-1}{2}$ and $x = 1$. Since $(1-\frac{q-1}{2})^{\frac{2-q}{q-1}} > 1-\frac{2-q}{q-1}\frac{q-1}{2} = \frac{q}{2} > \frac{1}{2}$, for $1 < q < \frac32$, $(1-\frac{q-1}{2})^{\frac{q}{q-1}} < e^{-\frac{q}{2}} < e^{-\frac12}$ and $\Gamma(1+\frac1q) <1$, we obtain
\[
f_q\left(1-\frac{q-1}{2}\right) > \frac{1}{(q-1)}\frac{1}{2}e^{-e^{-\frac12}} > \frac{1}{4(q-1)}.
\]
Moreover,
\[
f_q(1) = \frac{1}{(q-1)\Gamma(1+\frac1q)}e^{-1} > \frac{1}{e(q-1)} > \frac{1}{4(q-1)}.\qedhere
\]
\end{proof}

\begin{proof} [Finishing the proof of Proposition \ref{prop:main-proj}]
As earlier, Lemma \ref{lm:unif-under-proj} gives
\begin{equation} \label{eq:trape-proj}
\begin{split}
\mb{P}\big\{|X_1| \leq 1, \ \big||X_1|-|X_2|\big| <\alpha\big\} &\geq \int\!\!\!\!\int_{\{x \leq 1,\  |x-y| < \alpha\}}\left(\frac{1}{4(q-1)}\right)^2\1_{[1-\frac{q-1}{2},1]}(x)\1_{[1-\frac{q-1}{2},1]}(y)\, \dd x\, \dd y \\
&\geq \begin{cases} \frac{1}{64}, & \alpha > \frac{q-1}{2}, \\ \frac{\alpha}{32(q-1)}, & \alpha \leq \frac{q-1}{2}, \end{cases}
\end{split}
\end{equation}
where the last inequality follows from \eqref{eq:proba2} with $p$ replaced by $\tfrac{1}{q-1}$.

\noindent $\bullet$ If $\alpha > \frac{q-1}{2}$, inequalities \eqref{eq:proba-projec} and \eqref{eq:trape-proj} yield the lower bound
\[
\E\big(\alpha-|X|\big)_+  \geq \frac{\alpha}{25}\cdot \frac{1}{64}\cdot \frac{1}{2} = \frac{\alpha}{3200}.
\]
Since
\[
\alpha \stackrel{\eqref{eq:a-proje}}{\leq} c_q\sqrt{1-2a_2^2} \stackrel{\eqref{eq:estim-cq}}{<} 0.71\sqrt{1-2\left(\frac{1}{\sqrt{2}}-\frac{c}{p}\right)^2} < 0.71\sqrt{2\sqrt{2}\frac{c}{p}} < \frac{1.2}{\sqrt{L}},
\]
and $L \geq 8\, 294\, 400$,  we obtain the desired bound \eqref{eq:ultgoal-proj}.

\noindent $\bullet$ If $\alpha \leq \frac{q-1}{2}$, inequalities \eqref{eq:proba-projec} and \eqref{eq:trape-proj} give
\[
\E\big(\alpha-|X|\big)_+  \geq \frac{\alpha}{25}\cdot\frac{\alpha}{32(q-1)}\cdot\frac12 = \frac{\alpha^2}{1600(q-1)}.
\]
As we assume $1 - \frac{1}{q} \leq \frac{1}{cL+2}$, this is at least the desired $\frac{3}{4}\alpha^2$ by a large margin for $L = 8\, 294\, 400$. The proof is complete for every $1\leq q\leq q_0$, where
\[
q_0 = \frac{Lc+2}{Lc+1} > 1+ 10^{-12}.\qedhere
\]
\end{proof}


\section{Stability estimates with explicit constants} \label{sec:stab}


The proofs of both Theorems \ref{thm:dds} and \ref{thm:cnt} presented here follow the same strategy taken from \cite{CNT22}, which we shall now outline. For a unit vector $a$ in $\R^n$, consider again the deficit
\[
\delta(a) = \left|a - \frac{e_1+e_1}{\sqrt{2}}\right|^2 = 2-\sqrt{2}(a_1+a_2).
\]
Let $a$ be a unit vector and without loss of generality assume that $a_1\geq\cdots\geq a_n\geq0$. The approach leading to the stability of the inequalities of Szarek and Ball differs depending on whether the vector $a$ is \emph{close to} or \emph{far from} the extremizer $\frac{e_1+e_2}{\sqrt{2}}$, as measured by $\delta(a)$. 

\noindent \emph{Case 1.} When $a$ is close to $\tfrac{e_1+e_2}{\sqrt{2}}$, we quantitatively sharpen the inequalities of Szarek and Ball by reapplying them only to a \emph{portion} of the vector $a$, thus exhibiting their self-improving feature. The probabilistic formulae are crucial for this part. 

\smallskip

When $a$ is far from the extremizer, three things can happen. 

\smallskip

\noindent \emph{Case 2.} If the largest magnitude of the coordinates of $a$ is below $\tfrac{1}{\sqrt{2}}$, the second largest magnitude has to drop \emph{well} below $\frac{1}{\sqrt{2}}$ on the account of $\delta(a)$ being large and the classical Fourier-analytic approach of Haagerup \cite{Haa81} and Ball \cite{Bal86} allows to track the deficit.

\smallskip

\noindent \emph{Case 3.} If the largest magnitude is \emph{barely} above $\frac{1}{\sqrt{2}}$, a Lipschitz property of the section and projection functions allows to reduce this case to the one from Case 2.

\smallskip

\noindent \emph{Case 4.} If the largest magnitude is bounded below away from $\frac{1}{\sqrt{2}}$, an easy convexity/projection argument gives a strict inequality with a margin.


\subsection{Stability of Szarek's inequality} We first deal \mbox{with the sharp Khinchin inequality of \cite{Sza76}.}

\subsubsection*{Case 1.} We begin with the case that $a$ is near the extremizer.

\begin{lemma}\label{lm:stab-szarek}
Let $0 < \delta_0 < \frac23$ and take $c_0=\frac{1}{2\sqrt{2}}(\sqrt{\frac15(4-\delta_0)}-\sqrt{\delta_0})>0$. For every unit vector $a$ in $\R^n$ with $a_1 \geq \dots \geq a_n \geq 0$ satisfying $\delta(a) \leq \delta_0$, we have
\begin{equation}
\E\Big|\sum_{j=1}^n a_j\e_j\Big| \geq \frac{1}{\sqrt{2}} + c_0 \sqrt{\delta(a)}.
\end{equation}
\end{lemma} 

\begin{proof}
We will assume without loss of generality that $n \geq 3$ and $a_1^2+a_2^2 < 1$ (the remaining cases can be obtained by taking a limit). Let 
\[
\theta \eqdef \frac{1}{\sqrt{2}} \sqrt{1-a_1^2-a_2^2}.
\]
Arguing as in the induction of Section \ref{sec:proof-proj} and using Jensen's and Szarek's inequalities, we get
\begin{align*}
	\E\Big|\sum_{j=1}^n a_j \e_j \Big| & = \E \max\Big\{ |a_1 \e_1 + a_2 \e_2|, \Big| \sum_{j=3}^n a_j \e_j\Big| \Big\} \geq \E \max\Big\{ |a_1 \e_1 + a_2 \e_2|, \mb{E}\Big| \sum_{j=3}^n a_j \e_j\Big| \Big\}
	\\ & \geq \mb{E}\max\big\{ |a_1\e_1+a_2\e_2|,\theta\big\} = \frac12 \max\{a_1+a_2, \theta\} + \frac12 \max\{a_1-a_2, \theta\} \\
	& \geq \frac{a_1+a_2}{2} + \frac12 \sqrt{\frac15 (a_1-a_2)^2 + \frac45 \theta^2} 
	  = \frac{2-\delta(a)}{2\sqrt{2}} + \frac{1}{2\sqrt{2}} \sqrt{\frac15 (4-\delta(a))} \sqrt{\delta(a)} \\
	 & = \frac{1}{\sqrt{2}} + \frac{1}{2\sqrt{2}} \left( \sqrt{\frac15(4-\delta(a))} - \sqrt{\delta(a)} \right) \sqrt{\delta(a)} 
	  \geq \frac{1}{\sqrt{2}}+ c_0 \sqrt{\delta(a)},
\end{align*}
whenever $\delta(a)\leq \delta_0$.
\end{proof} 
 
\subsubsection*{Case 2.} We assume that $a$ is far from the extremizer and $a_1$ is at most $\frac{1}{\sqrt{2}}$. A key step in Haagerup's slick Fourier-analytic proof of Szarek's inequality from \cite{Haa81} is the bound
\begin{equation}\label{eq:Haa}
	\E \Big| \sum_{j=1}^n a_j \e_j \Big| \geq \sum_{j=1}^n a_j^2 F(a_j^{-2}), 
\end{equation}
for every unit vector $a$ in $\R^n$, where the function $F:(0,\infty)\to\R$ is given by
\[
F(s) = \frac{2}{\sqrt{\pi s}} \cdot \frac{\Gamma\left( \frac{s+1}{2} \right)}{\Gamma\left( \frac{s}{2} \right)}, \qquad s > 0.
\]
Haagerup showed that the function $F(s)$ is increasing on $(0,\infty)$, which will be crucial for us.
 
\begin{lemma}\label{lm:a1-small}
Let $0<\delta_0<2$. For every unit vector $a$ in $\R^n$ with $a_1 \geq \dots \geq a_n \geq 0$ satisfying $\delta(a) \geq \delta_0$ and $a_1 \leq \frac{1}{\sqrt{2}}$, we have
\begin{equation}
\E \Big| \sum_{j=1}^n a_j \e_j \Big| \geq \frac{1}{\sqrt{2}} + c_1\sqrt{\delta(a)},
\end{equation}
with $c_1 =  \frac{1}{2 \sqrt{2}} \left( F\left(  \frac{8}{(2-\delta_0)^2}  \right) - F(2) \right)$.
\end{lemma} 
 
\begin{proof}
We have
\[
	a_2 \leq \frac{a_1+a_2}{2} = \frac{2-\delta(a)}{2 \sqrt{2}} \leq \frac{2-\delta_0}{2 \sqrt{2}}, 
\]
which shows that $a_j^{-2} \geq l_0$, for all $j \geq2$, with $l_0 = \frac{8}{(2-\delta_0)^2}>2$.
Employing \eqref{eq:Haa} and using the monotonicity of $F$, we therefore have
\begin{align*}
	\E \Big| \sum_{j=1}^n a_j \e_j \Big| & \geq a_1^{2} F(2) + \sum_{j \geq 2} a_j^2 F(l_0) = a_1^{2} F(2) + (1-a_1^2) F(l_0) \\
	& = F(l_0) + a_1^{2}\big(F(2)-F(l_0)\big) \geq F(l_0) + \frac12\big(F(2)-F(l_0)\big) \\
	& = \frac12 \big(F(2)+ F(l_0)\big) = \frac{1}{\sqrt{2}} + \frac12 (F(l_0)-F(2)),
\end{align*}
since $F(2)=\tfrac{1}{\sqrt{2}}$. The conclusion follows since $\delta(a)\leq 2$.
\end{proof} 

\subsubsection*{Case 3.} We assume that $a$ is far from the extremizer but $a_1$ is barely larger than $\tfrac{1}{\sqrt{2}}$.
 
\begin{lemma}\label{lm:a1-moderate}
Let $\gamma_0 \leq 1- \frac{1}{\sqrt{2}}$ and $2\sqrt{\gamma_0}<\delta_0<2$. For every unit vector $a$ in $\R^n$ with coordinates $a_1 \geq \dots \geq a_n \geq 0$ satisfying $\frac{1}{\sqrt{2}} \leq a_1 \leq \frac{1}{\sqrt{2}} + \gamma_0 $ and $\delta(a) \geq \delta_0$, we have
\begin{equation}
\E \Big| \sum_{j=1}^n a_j \e_j \Big| \geq \frac{1}{\sqrt{2}} + c_2 \sqrt{\delta(a)},
\end{equation}
with
\begin{equation}
	c_2 = \frac{1}{2 \sqrt{2}} \left( F\left(  \frac{8}{(2+2 \sqrt{\gamma_0}-\delta_0)^2}  \right) - F(2) \right) \sqrt{\delta_0-2 \sqrt{\gamma_0}} - \sqrt{2\gamma_0+\gamma_0^2}.
\end{equation}
\end{lemma} 
 
\begin{proof}
We can assume that $c_2 \geq 0$ since otherwise it is enough to observe that
\[
	\E \Big| \sum_{j=1}^n a_j \e_j \Big| = \E \Big| a_1 + \sum_{j=2}^n a_j \e_j \Big|  \geq \Big| a_1 + \sum_{j=2}^n a_j  \E  \e_j \Big| = |a_1| \geq \frac{1}{\sqrt{2}}.
\]
Consider the unit vector 
\[
b\eqdef\left(\frac{1}{\sqrt{2}}, \sqrt{a_1^2+ a_2^2 - \frac12}, a_3, \ldots, a_n \right).
\]
Then by the triangle inequality, we obtain the following Lipschitz property,
\begin{align*}
	\E \Big| \sum_{j=1}^n a_j \e_j \Big| & \geq \E \Big| \sum_{j=1}^n b_j \e_j \Big| - \E \Big| \sum_{j=1}^n (a_j-b_j) \e_j \Big| \\ & \geq \E \Big| \sum_{j=1}^n b_j \e_j \Big| - \Big(\E \Big| \sum_{j=1}^n (a_j-b_j) \e_j \Big|^2\Big)^{1/2} 
	= \Big| \sum_{j=1}^n b_j \e_j \Big| - |a-b|.
\end{align*} 
Note that $b_1 \geq b_2$ and since $b_2 \geq a_2$, also $b_2 \geq b_3 \geq \dots \geq b_n$. Moreover,
\begin{equation} \label{eq:estim-ba}
\sqrt{a_1^2+a_2^2 - \frac{1}{2}} - a_2 = \frac{a_1^2-\frac{1}{2}}{\sqrt{a_1^2+a_2^2 - \frac{1}{2}} + a_2} \leq \sqrt{a_1^2-\frac{1}{2}} \leq \sqrt{\sqrt{2} \gamma_0 + \gamma_0^2}  < \sqrt{2\gamma_0},
\end{equation}
thus
\[
|a-b|^2 = \left(a_1 - \frac{1}{\sqrt{2}}\right)^2 + \left(\sqrt{a_1^2+a_2^2 - \frac{1}{2}} - a_2\right)^2 < \gamma_0^2 + 2\gamma_0.
\]
Observe that, since $a_1\geq\tfrac{1}{\sqrt{2}}$, we have
\begin{align*}
\delta(b) & = 2 - \sqrt{2}\left( \frac{1}{\sqrt{2}}+\sqrt{a_1^2+a_2^2 -\frac12} \right)  = \delta(a) - \sqrt{2}\left(\frac{1}{\sqrt{2}}+\sqrt{a_1^2+a_2^2 - \frac{1}{2}}-a_1-a_2\right)\\
&\geq \delta_0 -\sqrt{2}\left(\sqrt{a_1^2+a_2^2 - \frac{1}{2}}-a_2\right) \stackrel{\eqref{eq:estim-ba}}{>} \delta_0 - 2\sqrt{\gamma_0}.
\end{align*}
 Thus, applying Lemma \ref{lm:a1-small} to the vector $b$ and using the above estimates, we get
\begin{align*}
	\E \Big| \sum_{j=1}^n a_j \e_j \Big| & \geq \frac{1}{\sqrt{2}} +  \frac{1}{2 \sqrt{2}} \left( F\left(  \frac{8}{(2+2 \sqrt{\gamma_0}-\delta_0)^2}  \right) - F(2) \right) \sqrt{\delta(b)} - \sqrt{2\gamma_0+\gamma_0^2} \\
	& \geq \frac{1}{\sqrt{2}} +  \frac{1}{2 \sqrt{2}} \left( F\left(  \frac{8}{(2+2 \sqrt{\gamma_0}-\delta_0)^2}  \right) - F(2) \right) \sqrt{\delta_0-2 \sqrt{\gamma_0}} - \sqrt{2\gamma_0+\gamma_0^2}.
\end{align*}
Finally, as $a_1 \geq \frac{1}{\sqrt{2}}$, we have $\delta(a) = 2-\sqrt{2}(a_1+a_2) \leq 1-\sqrt{2}a_2 \leq 1$ and the proof is complete.
\end{proof}

\subsubsection*{Case 4.} Finally, there is also a simple bound for the case that $a_1$ is much larger than $\tfrac{1}{\sqrt{2}}$.

\begin{lemma}\label{lm:a1-large}
Let $\gamma_0>0$. For every unit vector $a$ in $\R^n$ with $a_1 \geq \frac{1}{\sqrt{2}}+ \gamma_0$, we have
\begin{equation}
	\E\Big| \sum_{j=1}^n a_j \e_j \Big| \geq \frac{1}{\sqrt{2}} + \gamma_0 \sqrt{\delta(a)}.
\end{equation}
\end{lemma} 
 
\begin{proof}
By Jensen's inequality and the fact that $\delta(a) \leq 1$,
\[
	\E\Big| \sum_{j=1}^n a_j \e_j \Big| \geq |a_1|   \geq  \frac{1}{\sqrt{2}} + \gamma_0  \geq   \frac{1}{\sqrt{2}} + \gamma_0 \sqrt{\delta(a)}.\qedhere
\]
\end{proof}

\subsubsection*{Constants.}

Combining Lemmas \ref{lm:stab-szarek}, \ref{lm:a1-small}, \ref{lm:a1-moderate} and \ref{lm:a1-large} with $\delta_0=0.66$ (\emph{almost} the maximal value allowed in Lemma \ref{lm:stab-szarek}) and $\gamma_0 = 8\cdot 10^{-5}$, we conclude that for all unit vectors $a$ in $\R^n$,
\[
\mb{E}\Big|\sum_{j=1}^n a_j\e_j\Big| \geq \frac{1}{\sqrt{2}}+ \kappa_1\sqrt{\delta(a)}
\]
with 
\[
\kappa_1 \geq \min\left\{c_0, c_1, c_2, \gamma_0\right\} > \min\big\{1.7\cdot 10^{-3}, 1.6\cdot 10^{-2}, 5.1\cdot 10^{-4}, 8\cdot10^{-5} \big\} = 8\cdot 10^{-5}.
\]
This completes the proof of Theorem \ref{thm:dds}.\hfill$\Box$


\subsection{Stability of Ball's inequality}

We now turn to the study of Ball's inequality \eqref{eq:ball}. Throughout this section we denote by $\QQ_n = [-\frac{1}{2},\frac{1}{2}]^n$ the cube of unit volume.

\subsubsection*{Case 1.} We begin with the case that $a$ is near the extremizer.

\begin{lemma}\label{lm:ball-stab-near}
For every $n \geq 2$ and every unit vector $a$ in $\R^n$ with $a_1 \geq \dots \geq a_n \geq 0$ satisfying $\delta(a) \leq \frac{1}{4}$, we have
\begin{equation}\label{eq:ball-stab1}
\vol\big(\QQ_n \cap a^\perp\big) \leq \sqrt{2} - c_1\sqrt{\delta(a)},
\end{equation}
where $c_1 = 0.12$.
\end{lemma}
\begin{proof}
We can assume that $n \geq 3$ and $a_1^2 + a_2^2 < 1$ (the missing cases follow by taking a limit). Leveraging a self-improving feature of Ball's inequality, the proof of \cite[Lemma 6.7]{CNT22} yields
\[
\vol\big(\QQ_n \cap a^\perp\big) \leq \sqrt{2}\max\left\{\left(1-\delta+\sqrt{\frac{\delta(2-\delta)}{5}}\right)^{-1}, (1-\delta)^{-2}\left(1-\delta-\frac{\sqrt{\delta(2-\delta)}}{2\sqrt{2}}\right)\right\},
\]
where $\delta = \delta(a)$. Denoting the maximum on the right-hand side by $M(\delta)$, we can take
\begin{equation*}
c_1 = \inf_{0 < \delta < 1/4}\sqrt{2}\frac{1-M(\delta)}{\sqrt{\delta}}.
\end{equation*}
Direct numerical calculations show that $c_1 > 0.12$.
\end{proof}

\subsubsection*{Cases 2 and 3.} We assume that $a$ is far from the extremizer but $a_1$ is not much larger than $\tfrac{1}{\sqrt{2}}$.

\begin{lemma}\label{lm:ball-stab-far1}
For every $n \geq 2$ and every unit vector $a$ in $\R^n$ with $a_1 \geq \dots \geq a_n \geq 0$ satisfying $\delta(a) \geq \frac{1}{4}$ and $a_1 \leq \frac{1}{\sqrt{2}}+\gamma_0$, we have
\begin{equation}\label{eq:ball-stab1}
\vol\big(\QQ_n \cap a^\perp\big) \leq \sqrt{2} - c_2,
\end{equation}
where $\gamma_0 = 3.2\cdot 10^{-5}$ and $c_2 = 0.0002$.
\end{lemma}
\begin{proof}
Here the proof relies on Fourier-analytic arguments. For  the special function
\begin{equation*}
\Psi(s) = \frac{2}{\pi}\sqrt{s}\int_0^\infty\left|\frac{\sin t}{t}\right|^s \dd t,
\end{equation*}
Ball showed in \cite{Bal86} that $\Psi(s) < \Psi(2) = \sqrt{2}$, for every $s > 2$. We need a robust version of this estimate.
Using the Nazarov--Podkorytov lemma \cite{NP00}, K\"onig and Koldobsky \cite{KK19} proved that
\begin{equation}\label{eq:KK-Psi}
\forall \ s\geq\frac{9}{4},\qquad \Psi(s) \leq \Psi(\infty) = \sqrt{\frac{6}{\pi}} = \sqrt{2}\left(\frac{3}{\pi}\right)^{1/2}
\end{equation}
(that is, $\theta_0 = \left(\frac{3}{\pi}\right)^{1/2}$ in the notation of \cite[Lemma~6.8]{CNT22}).
The argument now splits in two cases.

\smallskip

\noindent $\bullet$ \emph{Assume that $a_1 \leq \frac{1}{\sqrt{2}}$.} Provided that 
\[
s(a) \eqdef 2\Big(1-\frac{\delta(a)}{2}\Big)^{-2} \geq \frac94,
\]
which holds as long as $\delta(a) \geq 2\big(1-\frac{2\sqrt{2}}{3}\big) = 0.11..$,
with the aid of \eqref{eq:KK-Psi}, the arguments from \cite[Lemma~6.8]{CNT22} give the explicit estimate
\[
\vol\big(\QQ_n \cap a^\perp\big) \leq \left(\frac{3}{\pi}\right)^{1/4}\sqrt{2} = \sqrt{2} - \sqrt{2}(1-(3/\pi)^{1/4}).
\]
Therefore, we can take any
\begin{equation*}
c_2 \leq \sqrt{2}(1-(3/\pi)^{1/4}) = 0.016...
\end{equation*}

\smallskip

\noindent $\bullet$ \emph{Assume that $\frac{1}{\sqrt{2}} < a_1 \leq \frac{1}{\sqrt{2}} + \gamma_0$}.  Using Busemann's theorem \cite{Bus49}, this case is reduced in \cite[Lemma~6.8]{CNT22} to the previous range, which yields the bound
\[
\vol\big(\QQ_n \cap a^\perp\big) \leq \sqrt{2} - \sqrt{2}\min\left\{c_1\sqrt{\tfrac{1}{8} - \sqrt{\gamma_0}}, 1-(3/\pi)^{1/4}\right\} + 2\sqrt{\gamma_0^2+2\gamma_0},
\]
where $c_1$ is the constant from Lemma \ref{lm:ball-stab-near}.
With the choice of parameters $\gamma_0 = 3.2\cdot 10^{-5}$ and $c_1 = 0.12$, this estimate yields $\vol(\QQ_n \cap a^\perp)  \leq \sqrt{2} - 0.00021..$ and thus completes the proof. 
\end{proof}

\subsubsection*{Case 4.}
Finally, there is also a simple bound for the case that $a_1$ is much larger than $\tfrac{1}{\sqrt{2}}$.

\begin{lemma}\label{lm:ball-stab-far2}
For every $n \geq 2$ and every unit vector $a$ in $\R^n$ satisfying $a_1 \geq \frac{1}{\sqrt{2}}+\gamma_0$, we have
\begin{equation}\label{eq:ball-stab1}
\vol\big(\QQ_n \cap a^\perp\big) \leq \sqrt{2} - \frac{2\gamma_0}{1+\gamma_0\sqrt{2}}\sqrt{\delta(a)},
\end{equation}
where $\gamma_0=3.2\cdot10^{-5}$.
\end{lemma}
\begin{proof}
By Ball's geometric projection argument (see \cite{Bal86,NP00}), we have $\vol(\QQ_n \cap a^\perp) \leq \frac{1}{a_1}$. Since $a_1 \geq \frac{1}{\sqrt{2}}+\gamma_0$ and hence $\delta(a) < 1$, we deduce that
\[
\vol(\QQ_n \cap a^\perp) \leq \frac{1}{\frac{1}{\sqrt{2}}+\gamma_0} = \sqrt{2} - \sqrt{2}\left(1 - \frac{1}{1+\gamma_0\sqrt{2}}\right) \leq \sqrt{2} - \frac{2\gamma_0}{1+\gamma_0\sqrt{2}}\sqrt{\delta(a)}. \qedhere
\]
\end{proof}

\subsubsection*{Constants}

Combining Lemmas \ref{lm:ball-stab-near}, \ref{lm:ball-stab-far1} and \ref{lm:ball-stab-far2},  and using that always $\delta(a) < 2$, we conclude that for all unit vectors $a$ in $\R^n$, we have the inequality
\[
\vol\big(\QQ_n \cap a^\perp\big) \leq \sqrt{2} - \kappa_\infty\sqrt{\delta(a)}
\]
with 
\[
\kappa_\infty \geq \min\left\{c_1, \frac{c_2}{\sqrt{2}}, \frac{2\gamma_0}{1+\gamma_0\sqrt{2}}\right\} > 6\cdot 10^{-5}.
\]
This completes the proof of Theorem \ref{thm:cnt}.\hfill$\Box$

\begin{remark}
We would like to stress that the arguments of this paper have not been optimized to give the best possible constants $p_0$ and $q_0$ in Theorems \ref{thm:sec} and \ref{thm:proj}. We instead chose to be fairly generous in various parts of the proof for the sake of clarity of the exposition.
\end{remark}

\subsection*{Acknowledgments.} 
We should very much like to thank the anonymous referees for reading carefully our manuscript and sharing many valuable comments which helped significantly improve the paper.


\bibliographystyle{plain}
\bibliography{extremal-sec-proj}

\end{document}